\documentclass[reqno]{amsart}

\usepackage{amsmath}
\usepackage{amssymb}
\usepackage{mathdots}
\usepackage[bb=boondox]{mathalfa}
\usepackage{mathrsfs}
\usepackage{tikz}
\usepackage{mathabx}

\newif\ifstartedinmathmode
\newcommand\encircled[1]{%
  \relax\ifmmode\startedinmathmodetrue\else\startedinmathmodefalse\fi%
  \tikz[baseline,anchor=base]{%
  \node[draw=red,circle,outer sep=0pt,inner sep=.2ex]
    {\ifstartedinmathmode$#1$\else#1\fi};}%
}
\theoremstyle{plain}
\newtheorem{theorem}{Theorem}[section]

\newtheorem{lemma}[theorem]{Lemma}
\newtheorem{corollary}[theorem]{Corollary}
\newtheorem{proposition}[theorem]{Proposition}

\theoremstyle{definition}

\newtheorem{example}[theorem]{Example}
\newtheorem{question}[theorem]{Question}

\numberwithin{equation}{section}

\allowdisplaybreaks[1]

\newcommand{\BC}{{\mathbb C}}
\newcommand{\BF}{{\mathbb F}}
\newcommand{\BH}{{\mathbb H}}

\newcommand{\BL}{{\mathbb L}}
\newcommand{\BN}{{\mathbb N}}

\newcommand{\BR}{{\mathbb R}}

\newcommand{\cC}{{\mathcal C}}\newcommand{\cD}{{\mathcal D}}
\newcommand{\cE}{{\mathcal E}}
\newcommand{\cH}{{\mathcal H}}

\newcommand{\cL}{{\mathcal L}}

\newcommand{\cO}{{\mathcal O}}\newcommand{\cP}{{\mathcal P}}
\newcommand{\cR}{{\mathcal R}}
\newcommand{\cS}{{\mathcal S}}
\newcommand{\cV}{{\mathcal V}}
\newcommand{\cW}{{\mathcal W}}
\newcommand{\cZ}{{\mathcal Z}}

\newcommand{\fA}{{\mathfrak A}}
\newcommand{\fC}{{\mathfrak C}}

\newcommand{\fS}{{\mathfrak S}}\newcommand{\fT}{{\mathfrak T}}

\newcommand{\fY}{{\mathfrak Y}}

\newcommand{\wtilB}{\widetilde{B}}

\newcommand{\wtilK}{\widetilde{K}}

\newcommand{\whatA}{\widehat{A}}

\newcommand{\whatK}{\widehat{K}}

\newcommand{\al}{\alpha}
\newcommand{\be}{\beta}

\newcommand{\la}{\lambda}

\newcommand{\Up}{\Upsilon}

\newcommand{\rank}{\textup{rank\,}}

\newcommand{\kr}{\textup{Ker\,}}
\newcommand{\diag}{\textup{diag}}

\newcommand{\mat}[1]{\begin{bmatrix} #1 \end{bmatrix}}
\newcommand{\sbm}[1]{\left[\begin{smallmatrix} #1\end{smallmatrix}\right]}
\newcommand{\ov}[1]{{\overline{#1}}}
\newcommand{\un}[1]{{\underline{#1}}}

\newcommand{\tu}[1]{\textup{#1}}
\newcommand{\wtil}[1]{{\widetilde{#1}}}
\newcommand{\what}[1]{{\widehat{#1}}}

\newcommand{\ands}{\quad\mbox{and}\quad}

\newcommand{\BBone}{\mathbb{1}}
\newcommand{\OneVec}{\vec{\mathbf{1}}}

\newcommand{\vect}{\operatorname{vec}}

\begin{document}

\title[A Hill-Pick matrix criteria for the Lyapunov order]{A Hill-Pick matrix criteria for the Lyapunov order}

\author[S. ter Horst]{S. ter Horst}
\address{S. ter Horst, Department of Mathematics, Research Focus Area:\ Pure and Applied Analytics, North-West University, Potchefstroom, 2531 South Africa and DSI-NRF Centre of Excellence in Mathematical and Statistical Sciences (CoE-MaSS)}
\email{Sanne.TerHorst@nwu.ac.za}

\author[A. van der Merwe]{A. van der Merwe}
\address{A. van der Merwe, Faculty of Engineering and the Built Environment, Academic Development Unit, University of the Witwatersrand, Johannesburg, 2000 South Africa and DSI-NRF Centre of Excellence in Mathematical and Statistical Sciences (CoE-MaSS)}
\email{alma.naude1@wits.ac.za}

\thanks{This work is based on the research supported in part by the National Research Foundation of South Africa (Grant Number 90670 and 118513).}

\subjclass[2010]{Primary 93D30, 46L07, 47L07; Secondary 15A04, 15A39, 15B48, 47A57, 30E05, 15B05, 93D05}

%

\keywords{Lyapunov order, bicommutant of a matrix, Hill matrix, Pick matrix, positive matrix maps, completely positive matrix maps, Lyapunov equation, common Lyapunov solutions}

\begin{abstract}
The Lyapunov order appeared in the study of Nevanlinna-Pick interpolation for positive real odd functions with general (real) matrix points. For real or complex matrices $A$ and $B$ it is said that $B$ Lyapunov dominates $A$ if
\begin{equation*}
H=H^*,\quad   HA+A^*H \geq 0 \quad \implies \quad HB+B^*H \geq 0.
\end{equation*}
(In case $A$ and $B$ are real we usually restrict to real Hermitian matrices $H$, i.e., symmetric $H$.) Hence $B$ Lyapunov dominates $A$ if all Lyapunov solutions of $A$ are also Lyapunov solutions of $B$. In this paper we restrict to the case that appears in the study of Nevanlinna-Pick interpolation, namely where $B$ is in the bicommutant of $A$ and where $A$ is Lyapunov regular, meaning the eigenvalues $\lambda_j$ of $A$ satisfy
\begin{equation*}
    \lambda_i + \ov{\lambda}_j \ne 0, \quad i,j=1,\ldots,n.
\end{equation*}
In this case we provide a matrix criteria for Lyapunov dominance of $A$ by $B$. The result relies on a class of $*$-linear maps for which positivity and complete positivity coincide and a representation  of $*$-linear matrix maps going back to work of R.D. Hill. The matrix criteria asks that a certain matrix, which we call the Hill-Pick matrix, be positive semidefinite.
\end{abstract}

\maketitle

\section{Introduction}

Throughout this paper $\BF=\BC$ or $\BF=\BR$. To avoid confusion about transposes or adjoints, symmetric and Hermitian matrices, etc., we shall use notation and terminology as if $\BF=\BC$. We write $\cH_n$ for the set of $n \times n$ Hermitian matrices, $\cP_n$ for the set of all positive definite matrices in $\cH_n$ and $\ov{\cP}_n$ for the set of all positive semidefinite matrices in $\cH_n$. Given $A\in\BF^{n\times n}$, we call $H\in\cH_n$ a solution to the Lyapunov inequality of $A$ (Lyapunov solution of $A$ for short) if
\[
HA + A^*H \in \cP_n\ \ (\text{strict case}) \quad \text{or} \quad HA + A^*H \in \ov{\cP}_n\ \ (\text{non-strict case}).
\]
The corresponding solution sets are denoted
\[
\cH(A)=\{H \in \cH_n: HA + A^*H \in \cP_n\} \quad \text{and} \quad  \ov{\cH}(A)=\{H \in \cH_n: HA + A^*H \in \ov{\cP}_n\}.
\]
Note that these are {\em convex invertible cones}, that is, convex cones that are closed under inversion. Lyapunov inequalities play an important role in linear systems and control theory, cf., \cite{ZDG96,DP00,CL97}. The matrix $A$ is said to have {\em regular inertia} if $A$ has no eigenvalues on the imaginary axis $i\BR$. Moreover, we say that $A$ is {\em Lyapunov regular} if the eigenvalues $\la_1, \ldots, \la_n$ of $A$ satisfy
\[
\la_i+\ov{\la}_j \ne 0,\qquad i,j=1, \ldots,n.
\]

In \cite{CL09} Cohen and Lewkowicz introduced the {\em Lyapunov order} (for $\BF=\BR$), namely, for matrices $A,B \in \BF^{n \times n}$, it is said that $B$ {\em Lyapunov dominates} $A$, denoted $A \le_{\cL} B$, if
\begin{equation}\label{LyapOrder1}
H \in \cH_n, \quad HA+A^*H \in \ov{\cP}_n \implies HB+B^*H \in \ov{\cP}_n,
\end{equation}
or equivalently
\begin{equation}\label{LyapOrder2}
\ov{\cH}(A) \subseteq \ov{\cH}(B).
\end{equation}
Hence the Lyapunov order $A \le_\cL B$ means that all non-strict Lyapunov solutions of $A$ are also Lyapunov solutions of $B$.

More precisely, in \cite{CL09} the Lyapunov order was studied under some natural constraints on $A$ and $B$ that appear in the context of Nevanlinna-Pick interpolation, see also \cite{CL07}. For now restrict to $A,B\in\BR^{n \times n}$. Write $\cC(A)$ for the convex invertible cone generated by $A$, $\cP\cR\cO$ for the class of {\em positive real odd rational functions}, that is, rational functions with real coefficients that map the right half-plane into the right half-plane, and define $\cP\cR\cO(A)=\{ f(A) \colon f\in \cP\cR\cO \}$. Then $\cP\cR\cO(A)=\cC(A)$, and $B\in \cP\cR\cO(A)$ implies that $B$ is in the bicommutant $\{A\}_{\BR}''$ of $A$ and $A \le_{\cL} B$; cf., \cite{CL07,CL09} for further details. It is conjectured in \cite{CL09} that the converse implication also holds when $A$ is Lyapunov regular, but to the best of our knowledge there is no prove of this claim yet. In Subsection 1.4 of \cite{CL09} a Pick test to verify whether $A \le_{\cL} B$ is discussed. In the general setting, this Pick test corresponds to verifying whether a representative of each extreme ray of $\cH(A)$ is also in $\cH(B)$, while for $B\in\{A\}_{\BR}''$ it suffices to test this for a representative of a single extreme ray.

The main contribution of the present paper is that, by different methods, for both $\BF=\BC$ and $\BF=\BR$, given $B\in\{A\}_{\BF}''$ with $A$ Lyapunov regular, we determine a single matrix $\BH_{A,B}$ of size at most $n \times n$ so that $A \le_{\cL} B$ corresponds to $\BH_{A,B}$ being positive semidefinite. The matrix $\BH_{A,B}$ is obtained from what we called Hill representations for $*$-linear maps in \cite{tHvdM21b}, going back to work of R.D. Hill in \cite{H73}, and coincides with the classical Pick matrix in the case that $A$ and $B$ are diagonal matrices \cite{YS67}.

Determining when two matrices $A$ and $B$ have a common Lyapunov solution, that is, $\cH(A)\bigcap \cH(B) \ne \emptyset$ or $\ov{\cH}(A)\bigcap \ov{\cH}(B) \ne \emptyset$, is a notoriously difficult problem, which has only been resolved in special cases, cf., \cite{MS05,BG15,A01,H98,LS07,LS09,CL03} and references given there. While the inclusion condition $\ov{\cH}(A) \subseteq \ov{\cH}(B)$ of the Lyapunov order is a different problem, we expect that similar issues may arise when no further constraints are added; as mentioned above, the Pick test of \cite{CL09} for the general problem, with $A$ Lyapunov regular, involves determining the extreme rays of the Lyapunov solution set $\ov{\cH}(A)$. The essential condition in the present paper is that $B$ is in the bicommutant of $A$, which implies $B$ has a structure compatible with the Jordan structure of $A$; see Section \ref{LAP} below.

We further point out here that there is an analogous matrix order based on the Stein inequality $H-A^*HA\in \cP_n$, called the Stein order, which appears in the context of Nevanlinna-Pick interpolation in the unit disk, rather than the half-plane, which has also been extended to noncommutative several variable interpolation; see \cite{A04,BtH10,MS12,BMV18} for further details.

Given a square matrix $Y \in \BF^{n \times n}$, we define the {\em Lyapunov operator} associated with $Y$ by
\begin{equation}\label{LyapOp}
\cL_Y: \BF^{n \times n} \to \BF^{n \times n}, \quad \cL_Y(X)=XY + Y^*X, \quad X \in \BF^{n \times n}.
\end{equation}
The Lyapunov operator associated with $Y$ is a linear matrix map which is bijective precisely when $Y$ is Lyapunov regular, cf., \cite[Corollary 4.4.7]{HJ91}. It is clear that $\cL_Y(X^*)=\cL_ Y(X)^*$ for each $X\in \BF^{n \times n}$, so that $\cL_Y$ is $*$-linear (i.e., linear and preserves adjoints), so that, in particular, $\cL_Y$ maps $\cH_n$ into $\cH_n$; cf., \cite{tHvdM21b}.

Using the Lyapunov operator we can write the Lyapunov solution sets as
\[
\cH(A)=\cL_A^{-1}\left(\cP_n\right) \quad \text{and} \quad \ov{\cH}(A)=\cL_A^{-1}\left(\ov{\cP}_n\right).
\]
If $A\in\BF^{n \times n}$ is Lyapunov regular, $B\in\BF^{n \times n}$ Lyapunov dominates $A$ whenever the $*$-linear map
\begin{equation}\label{cL_AB}
\cL_{A,B}=\cL_B  \cL_A^{-1}: \BF^{n \times n}\to \BF^{n \times n}
\end{equation}
maps $\ov{\cP}_n$ into $\ov{\cP}_n$, that is, when $\cL_{A,B}$ is a positive linear matrix map. Determining whether a linear matrix map is positive is in general not easy. In contrast, it is easy to verify whether a linear matrix map is completely positive, by verifying that the associated Choi matrix is positive semidefinite. Complete positivity implies positivity, so that the Choi matrix criteria provides a sufficient condition, in general. However, one of the main results of this paper shows that in the specific case studied in this paper, the Choi matrix criteria is also necessary.

\begin{theorem}\label{main thm}
Let $A,B \in \BF^{n \times n}$ with $A$ Lyapunov regular and $B \in \{A\}_{\BF}''$ . Then
\begin{equation*}
\cL_{A,B} \text{ is positive if and only if } \cL_{A,B} \text{ is completely positive.}
\end{equation*}
Hence $\cL_{A,B}$ is positive if and only if the Choi matrix $\BL_{A,B} \in \BF^{n^2 \times n^2}$ of $\cL_{A,B}$ is positive semidefinite.
 \end{theorem}

This result will be proved in Section \ref{S:HillPick} using Theorem \ref{T:Main2} below, which provides a more general class of $*$-linear matrix maps for which positivity and complete positivity coincide, continuing a line of research initiated in  \cite{tHvdM21c}.

We now describe our main result in the setting of general $*$-linear matrix maps
\begin{equation}\label{cL-Intro}
\cL:\BF^{q\times q} \to \BF^{n\times n}.
\end{equation}
With such a map $\cL$ we associate two matrices, the Choi matrix $\BL$ given by
\begin{equation}\label{Choi}
\BL=\left[\BL_{ij}\right] \in \BF^{nq \times nq},\ \  \BL_{ij}=\cL\left(\mathcal{E}_{ij}^{(q)}\right) \in\BF^{n \times n}\ \ \mbox{for $i,j=1,\ldots,q,$}\end{equation}
where $\mathcal{E}_{ij}^{(q)}$ is the standard basis element in $\BF^{q \times q}$ with a 1 in position $(i,j)$ and zeros elsewhere, and what we call the matricization of $\cL$, which is the matrix $L\in \BF^{n^2 \times q^2}$ determined by the linear map
\begin{equation}\label{Matricization}
L:\BF^{q^2} \to \BF^{n^2},\quad L\,\left(\vect_{q}(V)\right)= \vect_{n}\left(\cL (V)\right),\quad V\in\BF^{q\times q},\end{equation}
where $\vect_{r \times s}:\BF^{r \times s} \to \BF^{rs}$ is the vectorization operator, abbreviated to $\vect_r$ in case $r=s$.

\begin{theorem}\label{T:Main2}
Let $\cL$ be a $*$-linear matrix map as in \eqref{cL-Intro} with matricization $L$ as in \eqref{Matricization}. If there exists a matrix $C \in \BF^{n \times n}$ such that $L$ is in $\ov{\{C\}_{\BF}''} \otimes \{C\}_{\BF}''$, then positivity and complete positivity of $\cL$ coincide.
\end{theorem}

Here $\ov{\{C\}_{\BF}''}$ is the algebra obtained by taking the (entrywise) complex conjugates of the elements of $\{C\}_{\BF}''$. We prove Theorem \ref{T:Main2}  in Subsection \ref{SubS:bicomm} below.

The condition that the matricization $L$ be contained in the algebra $\ov{\{C\}_{\BF}''} \otimes \{C\}_{\BF}''$ may seem restrictive, however, we point out that, by Theorem 5.1 in \cite{tHvdM21b}, the fact that $\cL$ is $*$-linear implies that there always exists a subspace $\cW\subset \BF^{n \times q}$ such that $L\in \ov{\cW} \otimes \cW$, for instance by taking $\cW$ equal to the span of the $n\times q$ blocks of $L$; see Corollary \ref{C:Lincl} below.

The proof of Theorem \ref{T:Main2} is based on a type of representation of linear matrix maps studied by R.D. Hill in \cite{H69,H73}, which we will refer to as Hill representations. These representations have the form
\begin{equation}\label{HillRepIntro}
\cL(V)=\sum_{k,l=1}^m \BH_{kl}\, A_k V A_l^*,\quad V\in\BF^{q \times q},
\end{equation}
for matrices $A_1,\ldots,A_m \in\BF^{n \times q}$. The matrix $\BH=\left[\BH_{kl}\right]_{k,l=1}^m\in\BF^{m\times m}$ is called the Hill matrix associated with the representation \eqref{HillRepIntro}. Moreover, we call the Hill representation \eqref{HillRepIntro} of $\cL$ minimal if $m$ is the smallest number of matrices $A_k$ that occurs in Hill representations for $\cL$, and it turns out that this smallest number equals the rank of the Choi matrix $\BL$ of $\cL$. A detailed analysis of minimal Hill representations was conducted in \cite{tHvdM21b} and a review of the relevant results from \cite{tHvdM21b} will be given in Section \ref{S:Hill}, along with some results on non-minimal Hill representations of $\cL$, which were not covered in \cite{tHvdM21b}.

If $\cL$ is given by the minimal Hill representation \eqref{HillRepIntro}, then the Choi matrix factors as
\begin{equation}\label{ChoiFact}
\BL=\whatA^* \BH^T \whatA\quad \mbox{with } \widehat{A}^*:=\begin{bmatrix} \vect_{n \times q}\left(A_1\right) & \hdots & \vect_{n \times q}\left(A_m\right)  \end{bmatrix}\in \BF^{nq \times m}
\end{equation}
and $\whatA$ is a matrix with full row rank. As proved by Poluikis and Hill in \cite{PH81}, it follows that complete positivity of $\cL$ corresponds to positive definiteness of $\BH$. Positivity of $\cL$ corresponds to\begin{equation}\label{PosCon}\left(z\otimes x\right)^* \BL \left(z\otimes x\right) \ge 0 \quad \mbox{for all}\quad x\in\BF^n \text{ and } z \in \BF^q,\end{equation} see \cite{KMcCSZ19}. In particular, by the above factorisation of $\BL$, it follows that a positive map $\cL$ is completely positive when the bilinear map determined by the matrix $\whatA$:
\begin{equation}\label{Bilinear}
(z,x)\mapsto \whatA (z\otimes x),\quad x\in\BF^n,\, z \in \BF^q
\end{equation}
is surjective. 
Little appears to be known about the ranges of bilinear maps. In \cite{tHvdM21c} we considered special cases for which \eqref{Bilinear} is surjective, although we also encountered a case where the bilinear map \eqref{Bilinear} was not surjective, while positivity and complete positivity of $\cL$ still coincided. An instinctive approach we use in this paper is to find a vector $x \in \BF^n$ such that $\rank \widehat{A}\left(I_q \otimes x\right)=m,$ or a vector $z \in \BF^q$ such that $\rank \widehat{A}\left(z \otimes I_n\right)=m$, since for both cases it follows that \eqref{Bilinear} is surjective; with $I_p$ the identity matrix of size $p \times p$. We provide a criteria for when this happens, which includes the case of Theorem \ref{T:Main2}.

The Choi matrix $\BL$ of $\cL$ is of size $nq \times nq$, but its rank $m$ may be much smaller. Hence to test for complete positivity, it may be more convenient to test whether the Hill matrix $\BH$ is positive definite than to test positive semidefiniteness of $\BL$. Moreover, as shown in examples in \cite{tHvdM21b} and exploited further in \cite{tHvdM21c}, by making an appropriate selection of the Hill representation, in the factorization of the Choi matrix \eqref{ChoiFact}, typically `structural properties' of $\cL$ (more correctly, of $L$) are contained in the matrix $\whatA$ while the `data' of $\cL$ is stored in $\BH$. To some extend, this phenomenon also occurs in non-minimal Hill representations, provided the choice is made in the right way; see Subsection \ref{SubS:NonMinHill}.

In the context of the $*$-linear map $\cL_{A,B}$, with $A$ and $B$ as in Theorem \ref{main thm}, it turns out that the rank $m$ of the Choi matrix is at most $n$, so that the Hill matrix $\BH_{A,B}$, called the Hill-Pick matrix for this specific case, has a size of at most $n \times n$. We compute this matrix explicitly, as well as the matricization $L_{A,B}$ determined by $\cL_{A,B}$, in Section \ref{S:HillPick}. For the case where $A$ and $B$ are diagonal, $\BH_{A,B}$ coincides with the Pick matrix from \cite{YS67}, hence our choice of the name Hill-Pick matrix.

Together with the current introduction, the paper consists of five sections. We provide some linear algebra preliminaries in Section \ref{LAP}, focusing on the  bicommutant of a matrix based on its Jordan decomposition, for both $\BF=\BC$ and $\BF=\BR$. In Section \ref{S:Hill} we recall some results on $*$-linear maps and minimal Hill representations from \cite{tHvdM21b} that will be used throughout the paper. We also prove some results on non-minimal Hill representations in this section. Then we present some new results on classes of $*$-linear matrix maps for which positivity and complete positivity coincide in Section \ref{pos imply completely pos}, extending our research from \cite{tHvdM21c}. In particular, in Section \ref{pos imply completely pos} we prove Theorem \ref{T:Main2}. The proof of our main result on the Lyapunov order, Theorem \ref{main thm}, will be given in Section \ref{S:HillPick} along with explicit formulas for the Hill-Pick matrix $\BH_{A,B}$ for the case where $\BF=\BC$.

\section{Linear algebra preliminaries}\label{LAP}

In this section we present some notation and preliminary results from linear algebra that can mostly be found in the standard literature \cite{HJ85,HJ91,W16}. For some more specialized results we do provide precise references or a proof in case we did not find an appropriate source. There is some overlap with notation introduced in the introduction, but we feel it may be helpful to the reader to have everything in one place.

\subsection{Notation and basic linear algebra results}

Throughout this paper $\BF=\BC$ or $\BF=\BR$. We write $\BF^{n \times m}$ for the vector space of $n \times m$ matrices over $\BF$ and $\BF^n$ for the space of all (column) vectors over $\BF$ of length $n$. Occasionally we will identify $\BF^n$ with $\BF^{n \times 1}$, so that matrix operations can be applied to vectors in $\BF^n$. Furthermore we write $\textup{GL}(n,\BF)$ for the set of invertible matrices in $\BF^{n \times n}$.

The standard $j$-th basis element in $\BF^n$ is denoted by $e_j^{(n)}$ or simply $e_j$ when the length is clear from the context. We write $\cE_{ij}^{(n,m)}$ for the standard basis element of $\BF^{n \times m}$ with $1$ in position $(i,j)$ and zeros elsewhere, i.e., $\cE_{ij}^{(n,m)}=e_i^{(n)}e_j^{(m)T}$, abbreviated to $\cE_{ij}^{(n)}$ when $m=n$. With $\OneVec_n$ we indicate the all-one vector of length $n$ and with $\BBone_{n \times m}$ the all-one matrix of size $n \times m$, so that $\BBone_{n \times m}=\OneVec_{n}\OneVec_m^T$. Also here, we write $\BBone_n$ for $\BBone_{n \times n}$. Furthermore, $I_n$ denotes the $n \times n$ identity matrix.

For $A \in \BF^{n \times m}$ we write $A^T$ for its transpose, $A^*$ for its adjoint and $\overline{A}$ for its complex conjugate. Of course, if $\BF=\BR$, then $A^T=A^*$ and $\ov{A}=A$. The null space of $A$ is denoted $\kr A$. Whenever $\Lambda \subset \BF^{n \times m}$ is a subset of matrices, then by $\ov{\Lambda}$, $\Lambda^T$, $\Lambda^*$, etc.\ we indicate the sets of matrices obtained by applying the appropriate operation to the matrices in $\Lambda$.

We write $\cH_n$ for the real subspace of Hermitian matrices in $\BF^{n \times n}$, which coincides with the $n \times n$ symmetric matrices $\cS_n$ in case $\BF=\BR$, but not if $\BF=\BC$.  With $A\geq 0$ (resp.\ $A>0$) we indicate that $A$ is positive semidefinite (resp.\ positive definite), and $\ov{\cP}_n$ (resp.\ $\cP_n$) denotes the set of positive semidefinite (resp. positive definite) matrices in $\cH_n$.

The {\em Kronecker product} of matrices $A=\left[a_{ij}\right]\in\BF^{n \times m}$ and $B\in\BF^{k \times l}$ is defined as
\[
A\otimes B=\left[a_{ij}B\right]\in \BF^{(n k) \times (m l)}.
\]
Note that for matrices $A$, $B$, $C$ and $D$ of appropriate size we have
\begin{equation} \label{kron prop}
\left(A \otimes B\right)\left(C \otimes D \right)=(AC) \otimes (BD).
\end{equation}
The \emph{vectorization} of a matrix $T \in \BF^{n \times m}$ is the vector $\vect_{n \times m }{(T)} \in \BF^{nm}$ defined as \[\vect_{n \times m}(T)= \sum_{j=1}^m\left(e_j^{(m)} \otimes I_n\right)Te_j^{(m)}.\]
Note that the vectorization operator $\vect_{n \times m}$ defines an invertible linear map from $\BF^{n\times m}$ onto $\BF^{nm}$ whose inverse is indicated by $\vect_{n \times m}^{-1}$. If $m=n$ we just write $\vect_n$ and $\vect_n^{-1}$ and if the sizes are clear from the context, the indices are often left out.  Furthermore, we have the identity
\[
\vect\left(A XB^T\right)=(B \otimes A)\vect(X)
\]
which, when $A$ and $B$ are taken to be (transposes of) vectors yields
\begin{equation}\label{VectForm1}
(z \otimes x)^T \vect_{n \times m}(W)=x^T W z, \quad W \in \BF^{n \times m}, x \in \BF^n, z \in \BF^m.
\end{equation}

\noindent Moreover note that
\[
\vect_{n \times m}\left(v w^T\right)=w \otimes v,\quad v\in\BF^n,w\in\BF^m,
\]
from which we obtain that
\begin{equation}\label{UnitVectId1}
\vect_{m \times n}\left(\mathcal{E}_{lk}^{(m,n)}\right)=\vect_{m \times n}\left(e_l^{(m)}e_k^{(n)^T}\right)= e_k^{(n)} \otimes e_l^{(m)} =e_{(k-1)m +l}^{(nm)}.
\end{equation}
Next, recall that the {\em Hadamard product} of matrices $A=\left[a_{ij}\right], B=\left[b_{ij}\right]\in\BF^{n \times m}$ is defined as
\[A\circ B=\left[a_{ij}b_{ij}\right]\in\BF^{n \times m}.
\]
The {\em canonical shuffle} is the matrix defined by
\begin{equation*}
\fS_{mn}:=\sum_{i=1}^m\sum_{j=1}^n \mathcal{E}_{ij}^{(m,n)} \otimes \mathcal{E}_{ji}^{(n,m)} \in \BF^{mn \times mn},
\end{equation*}
which occasionally will be identified with a linear map on $\BF^{mn}=\BF^{nm}$. Note that $\fS_{mn}$ is an invertible map with inverse $\fS_{mn}^{-1}=\fS_{mn}^*=\fS_{nm}$ and that for $u\in \BF^{m}$, $v \in \BF^{n}$ and $A \in \BF^{m \times n}$ we have
\begin{equation}\label{CanShuffleIds}
\fS_{mn}\vect_{n \times m}\left(A^T\right)=\vect_{m \times n}\left(A\right),\quad
\fS_{mn}\left(u \otimes v\right)=v \otimes u.
\end{equation}

We conclude this subsection with a lemma of a general linear algebra nature, that will be of use in the sequel.

\begin{lemma}\label{L:MatTrans}
Let $K_1,\ldots, K_m \in\BF^{n \times q}$ be linearly independent. Define
\begin{equation}\label{MatTrans}
\wtilK=\mat{\ov{K}_1\\\vdots\\ \ov{K}_m}\in\BF^{mn \times q}, \, \, \widecheck{K}=\mat{K_1^*\\\vdots\\ K_m^*}\in\BF^{mq \times n}  \text{ and }
\whatK=\mat{\widehat{K}_1 \\ \vdots \\ \widehat{K}_m}\in\BF^{m \times nq}
\end{equation}
where $\whatK_l:=\vect_{n \times q}\left(\ov{K}_l\right)^{T}$ for $l=1,\ldots,m$. Then $\rank \whatK=m$, hence $\whatK$ has full row rank, and for $x\in\BF^n, z \in \BF^q$ we have\begin{equation}\label{KtilKhatRel}  \whatK \left(z \otimes I_n \right)=\left(I_m\otimes z\right)^T \widecheck{K} \ands \whatK \left(I_q \otimes x \right)=\left(I_m\otimes x\right)^T \wtilK.  \end{equation}
\end{lemma}

\begin{proof}[\bf Proof]
Using \eqref{VectForm1}, we obtain the first identity from
\begin{align*}
\widehat{K}\left(z \otimes I_n\right)&=\mat{\vect_{n\times q}\left(K_1\right)^*\left(z \otimes I_n\right)\\\vdots \\ \vect_{n \times q}\left(K_m\right)^*\left(z \otimes I_n\right)}=\begin{bmatrix}z^TK_1^* \\ \vdots \\ z^TK_m^* \end{bmatrix}=\left(I_m \otimes z\right)^T\widecheck{K}.
\end{align*}
The second identity as well as the remaining claims follow from Lemma 5.5 in \cite{tHvdM21b} (where in the definition of $\wtilK$ the bars are not included), noting that the context in which they were proved there is not relevant to the argument.
\end{proof}

\subsection{Real and complex Jordan decomposition} The (complex) Jordan form of a matrix $A\in\BC^{n\times n}$ with (distinct) eigenvalues $\la_1,\ldots,\la_r$ is the decomposition of $A$ of the form
\begin{equation}\label{JordanComp}
A=P\, J_A\, P^{-1},\quad \mbox{with }J_A=\diag\left(J_{\un{n}_1}\left(\la_1\right),\ldots,J_{\un{n}_r}\left(\la_r\right)\right)
\end{equation}
where $P\in\textup{GL}(n,\BC)$ and $\un{n}_j=\left(n_{j,1},\ldots,n_{j,k_j}\right)\in\BN^{k_j}$, ordered decreasingly, for some $k_j\in\BN$ and for $j=1,\ldots,r$, and where we set
\[
J_{\un{n}_j}\left(\la_j\right)=\diag\left(J_{n_{j,1}}\left(\la_j\right), \ldots,J_{n_{j,k_j}}\left(\la_j\right)\right),
\]
with $J_{n}(\la)$ denoting the $n \times n$ Jordan matrix with eigenvalue $\la$:
\begin{equation}\label{Jordan}
J_{n}(\la)= \la I_n+ S_n,\quad \mbox{where}\quad S_n=[\delta_{i+1,j}]_{i,j=1}^n,
\end{equation}
with $\delta_{i,j}$ the Kronecker delta symbol. Note that $S_n$ is the upper shift matrix. In this form the matrix $J_A$ is uniquely determined by $A$. In case $A$ is a real matrix, non-real eigenvalues come in complex conjugate pairs with corresponding Jordan structure, i.e., when $J_{\un{n}}(\la)$ occurs in $J_A$ with $\la \in \BC$ non-real, then also $J_{\un{n}}\left(\ov{\la}\right)$ occurs in $J_A$.

In order to introduce the real Jordan form of a real matrix $A$, let $\fC\subset \BR^{2\times 2}$ denote the $2\times 2$ matrix representation of the complex numbers, that is, $\fC$ is the commutative subalgebra of $\BR^{2\times 2}$  given by
\[
\fC=\left\{C_\la:=\sbm{a&b\\-b&a}\in\BR^{2\times 2} \colon \la=a+ib\in\BC \right\}.
\]
With $C\in\fC$ and $n\in\BN$ we associate the real Jordan matrix
\begin{equation}\label{RealJordan}
J_n(C)=I_n\otimes C +S_{n} \otimes I_2,
\end{equation}
and for a tuple $\un{n}=\left(n_1,\ldots,n_k\right)\in\BN^k$ we set
\[
J_{\un{n}}(C)=\diag\left(J_{n_1}(C),\ldots, J_{n_k}(C)\right).
\]

Given a real matrix $A\in\BR^{n\times n}$ with (distinct) complex eigenvalues $\la_1,\ldots,\la_{r_1}$ and (distinct) real eigenvalues $t_1,\ldots,t_{r_2}$, in addition to its complex Jordan form described above, $A$ also has a real Jordan form which is the decomposition of $A$ of the form $A=P\,J_A\,  P^{-1},$ where
\begin{equation}\label{JordanReal}
 J_A=\diag\left(J_{\un{m}_1}\left(C_{\la_1}\right),\ldots,J_{\un{m}_{r_1}}\left(C_{\la_{r_1}}\right),J_{\un{n}_1}\left(t_1\right),\ldots , J_{\un{n}_{r_2}}\left(t_{r_2}\right)\right)
\end{equation}
with $P\in\textup{GL}(n,\BR)$ and $\un{n}_j=\left(n_{j,1},\ldots,n_{j,k_j}\right)\in\BN^{k_j}$, $\un{m}_s=\left(m_{s,1},\ldots,m_{s,l_s}\right)\in \BN^{l_s}$, both ordered decreasingly, for some $k_j,l_s\in\BN$ and for $s=1,\ldots,r_1$ and $j=1,\ldots,r_2$.

\subsection{Block Toeplitz matrices}
Let $\fA\subset \BF^{m \times m}$ be a matrix algebra. We write $\fT_{n,\fA}$ for the class of $n\times n$ block Toeplitz matrices with blocks from $\fA$, that is, the matrices of the form
\[
T=\sum_{j=1}^{n-1} S_n^{*j}\otimes A_{-j} + \sum_{j=0}^n S_n^j\otimes A_j\quad \mbox{with}\quad A_{-n+1},\ldots,A_{n}\in\fA,
\]
with $S_n$ the upper shift matrix defined in \eqref{Jordan}. Furthermore, $\fT_{n,\fA}^+$ denotes the matrix algebra of upper triangular block Toeplitz matrices in $\fT_{n,\fA}$, i.e., with $A_{-n+1}=\cdots=A_{-1}=0$, and $\fT^{-}_{n,\fA}$ the matrix algebra of lower triangular block Toeplitz matrices in $\fT_{n,\fA}$, i.e., $A_1=\cdots=A_n=0$. We will encounter block Toeplitz matrices with various choices of $\fA$; apart from $\BC$ and $\BR$ there will also be instances where $\fA$ is equal to $\fC$ or $\fT_{n,\BC}^+$, $\fT_{n,\BR}^+$ or $\fT_{n,\fC}^+$. Note that when $\fA$ is a commutative algebra, then so are $\fT_{n,\fA}^+$ and $\fT_{n,\fA}^-$.

Note that the Jordan matrix $J_n(\la)$ in \eqref{Jordan} is in $\fT_{n,\BC}^+$, while the real Jordan matrix $J_n(C)$ in \eqref{RealJordan} is in $\fT_{n,\fC}^+$. To accommodate the matrices $J_A$ appearing in the real and complex Jordan decompositions, we also introduce a class of repeated, compressed block Toeplitz matrices. For $\un{n}=\left(n_1,\ldots,n_k\right)\in\BN^k$ we write $\fT_{\un{n},\fA}$ for the classes of block diagonal matrices of the form $T=\diag\left(T_1,\ldots,T_k\right)$ with $T_j\in\fT_{n_j,\fA}$ and for $1\leq j,s\leq k$ with $n_j\geq n_s$ we have that $T_s$ is the compression of $T_j$ to the first $n_s$ (block) columns and first $n_s$ (block) rows. Similarly we define $\fT_{\un{n},\fA}^+$ by restricting the block diagonal entries to be upper triangular block Toeplitz matrices and $\fT^{-}_{\un{n},\fA}$ is defined by restricting the block diagonal entries to be lower triangular block Toeplitz matrices. Note that $\fT_{\un{n},\fA}^+$ and $\fT^{-}_{\un{n},\fA}$ are also algebras, which are commutative whenever $\fA$ is commutative.

The following lemma will be of use in the next subsection.

\begin{lemma} \label{result on tensors and algebras}
For $\un{n}\in\BN^k$ and $\fA_1,\fA_2$ matrix algebras over $\BF$, we have
\[
\fT^+_{\un{n},\fA_1}\otimes\fA_2=\fT^+_{\un{n},\fA_1\otimes \fA_2}.
\]
\end{lemma}

\begin{proof}[\bf Proof]
Note that $\fT^+_{\un{n},\fA_1}=\fT^+_{\un{n},\BF}\otimes \fA_1$. By the associativity property of tensors we find that
\[
\fT^+_{\un{n},\fA_1}\otimes\fA_2=\left(\fT^+_{\un{n},\BF}\otimes \fA_1\right)\otimes \fA_2= \fT^+_{\un{n},\BF} \otimes \left(\fA_1\otimes \fA_2\right)=\fT^+_{\un{n},\fA_1\otimes\fA_2}.\qedhere
\]
\end{proof}

\subsection{The real and complex bicommutant of a matrix}\label{SubS:bicommA}

For any subset $\cD \subseteq \BF^{n \times n}$, the commutant of $\cD$ in $\BF^{n \times n}$ is the matrix algebra, which is closed under inversion, given by
\[
\mathcal{D}'_{\BF} = \{ C \in \BF^{n \times n}: CD = DC \quad \text{for all} \quad D \in \mathcal{D}\}.
\]
The subscript $\BF$ is added since if $\cD$ consists of real matrices only, one can consider the commutant both in $\BC^{n \times n}$ and $\BR^{n \times n}$. The bicommutant of $\cD$ is the commutant of $\mathcal{D}'_{\BF}$ in $\BF^{n \times n}$  and denotes as $\cD_\BF''$.

If $\cD=\{A\}$ for a matrix $A\in \BF^{n \times n}$, then $\{A\}_{\BF}''\subset \{A\}_\BF'$ which implies $\{A\}_\BF''$ is a commutative algebra which is closed under inversion. The bicommutant $\{A\}_\BF''$ is obtained by applying polynomials to $A$:
\[
\{A\}_\BF''=\left\{p(A)\colon p\in\BF[x]\right\}.
\]
See Chapters 5 and 6 in \cite{C66} for more details.

For the purpose of this paper we are interested in a representation of the elements of $\{A\}_\BC''$ and $\{A\}_\BR''$ in terms of the Jordan forms of $A$. Note that \begin{align*}\{J_n(\la)\}_\BC''&=\fT_{n,\BC}^+ \text{ for all } \la\in\BC,\quad \{J_n(t)\}_\BR''=\fT_{n,\BR}^+ \text{ for all } t\in\BR \quad \text{and} \\ \{J_n(C)\}_\BR''&=\fT_{n,\fC}^+ \text{ for all } C\in\fC.\end{align*}
Furthermore, when a polynomial is applied to a direct sum of Jordan blocks with the same eigenvalue one obtains a direct sum of repeated compressions of the same Toeplitz matrix while for different eigenvalues, polynomials can be selected in such a way that the Toeplitz matrices can be chosen independently. Put together, we find for $A\in\BC^{n\times n}$ given in Jordan form \eqref{JordanComp} that $B\in\BC^{n\times n}$ is in $\{A\}_\BC''$ if and only if
\[
B=P\, \diag\left(T_1,\dots,T_r\right)\, P^{-1},\quad \mbox{with} \quad T_j\in\fT_{\un{n}_j,\BC}^+,\quad j=1,\ldots, r.
\]
Hence, we have
\begin{equation}\label{AbicomC}
\{A\}_\BC''= P\, \diag\left(\fT_{\un{n}_1,\BC}^+,\dots,\fT_{\un{n}_r,\BC}^+\right)\, P^{-1}=P\,\{J_A\}_{\BC}''\,P^{-1}.
\end{equation}
Note that \[\{A^T\}_{\BC}''=\left(P^{-1}\right)^T\, \diag\left(\fT_{\un{n}_1,\BC}^-,\dots,\fT_{\un{n}_r,\BC}^-\right)\,P^T =\left(P^{-1}\right)^T\,\{J_A^T\}_{\BC}''\,P^T\] and \[
\{J_A^*\}_{\BC}''=\{\ov{J_A}^T\}_{\BC}''=\ov{\{J_A^T\}_{\BC}''}=\{J_A^T\}_{\BC}''.
\]
The latter identities do not necessarily hold if $J_A$ is replaced by $A$.

Likewise, for $A\in\BR^{n\times n}$ given in real Jordan form \eqref{JordanReal}, a matrix $B\in\BR^{n\times n}$ is in $\{A\}_\BR''$ if and only if
\[
B=
P\,\diag\left(R_1,\ldots,R_{r_1},T_1,\dots , T_{r_2}\right)\,  P^{-1}, \mbox{ with }
R_s\in\fT_{\un{m}_s,\fC}^+, \quad
T_j\in\fT_{\un{n}_j,\BR}^+,
\] for all $\ s=1,\dots,r_1$ and  $j=1,\ldots, r_2.$
Hence, we have
\begin{equation}\label{AbicomR}
\{A\}_\BR''= P\, \diag\left(\fT_{\un{m}_1,\fC}^+,\dots,\fT_{\un{m}_{r_1},\fC}^+,\fT_{\un{n}_1,\BR}^+,\dots,\fT_{\un{n}_{r_2},\BR}^+\right)\, P^{-1}=P\,\{J_A\}_{\BR}''\,P^{-1}
\end{equation}
and
\begin{align*}
\{A^T\}_\BR''&= \left(P^{-1}\right)^T\, \diag\left(\fT_{\un{m}_1,\fC}^-,\dots,\fT_{\un{m}_{r_1},\fC}^-,\fT_{\un{n}_1,\BR}^-,\dots,\fT_{\un{n}_{r_2},\BR}^-\right)\, P^T\\&=\left(P^{-1}\right)^T\,\{J_A^T\}_{\BR}''\,P^T.
\end{align*}

Lastly, for a given matrix $A \in \BF^{n \times n}$ we describe the structure of $\{A\}_\BC'' \otimes \{A\}_\BC''$ and in case $A$ is a real matrix also the structure of $\{A\}_\BR'' \otimes \{A\}_\BR''$. In both cases the description is based on the Jordan structures of $A$. We start with the complex case.

\begin{proposition}\label{P:SelfTensCdoubleCom}
Let $A\in\BC^{n\times n}$ be given in Jordan form \eqref{JordanComp}. Then
\[
\{A\}_{\BC}'' \otimes \{A\}_\BC''= (P\otimes P)\, \diag\left(\fT^+_{\un{n}_1,\{J_A\}_\BC''},\dots, \fT^+_{\un{n}_r,\{J_A\}_\BC''}\right)\, (P\otimes P)^{-1}.
\]
\end{proposition}

\begin{proof}[\bf Proof]
Making use of the formula for $\{A\}_{\BC}''$ in \eqref{AbicomC} with \eqref{kron prop} it follows that
\begin{equation*}
    \begin{aligned}
    \{A\}_{\BC}'' \otimes \{A\}_{\BC}'' &=\left(P\, \{J_A\}_{\BC}'' \,P^{-1}\right) \otimes \left(P\, \{J_A\}_{\BC}'' \,P^{-1}\right) \\&=(P \otimes P)\,\left(\{J_A\}_{\BC}'' \otimes \{J_A\}_{\BC}''\right) \,(P\otimes P)^{-1}\\ &= \left( P \otimes P\right)\, \left(\diag\left(\fT_{\un{n}_1,\BC}^+,\dots,\fT_{\un{n}_r,\BC}^+\right)\otimes \{J_A\}_{\BC}''\right)\,\left( P  \otimes P\right)^{-1} \\
&= \left( P \otimes P\right)\, \diag\left(\fT_{\un{n}_1,\BC}^+\otimes \{J_A\}_{\BC}'',\dots \right. \\ &\left. \qquad \qquad \qquad \dots,\fT_{\un{n}_r,\BC}^+\otimes \{J_A\}_{\BC}''\right)\,\left( P  \otimes P\right)^{-1}\\
&= \left( P \otimes P\right)\, \diag\left(\fT_{\un{n}_1,\{J_A\}_{\BC}''}^+,\dots, \fT_{\un{n}_r,\{J_A\}_{\BC}''}^+\right)\,\left( P  \otimes P\right)^{-1},
\end{aligned}
\end{equation*}
where in the last identity we applied Lemma \ref{result on tensors and algebras} with $\fA_1=\BC$ and $\fA_2=\{J_A\}_{\BC}''$.
\end{proof}

\noindent Applying the same arguments as in the proof of Proposition \ref{P:SelfTensCdoubleCom} one obtains 
\begin{align} \label{formula for transpose C}
    \{A^T\}_{\BC}'' \otimes \{A^T\}_{\BC}''&=\left(P^T \otimes P^T\right)^{-1} \,\left(\{J_A^T\}_{\BC}'' \otimes \{J_A^T\}_{\BC}''\right)\,\left( P^T  \otimes P^T\right) \\&= \left(P^T \otimes P^T\right)^{-1}\, \diag\left(\fT_{\un{n}_1, \{J_A^T\}_{\BC}''}^-,\dots,\fT_{\un{n}_r, \{J_A^T\}_{\BC}''}^-\right)\,\left( P^T  \otimes P^T\right)\nonumber
\end{align} as well as
\begin{align*}
\{A\}_{\BR}'' \otimes \{A\}_\BR''&=\left(P \otimes P\right) \,\left(\{J_A\}_{\BR}'' \otimes \{J_A\}_{\BR}''\right)\,\left( P  \otimes P\right)^{-1}\\&= (P\otimes P)\,
\diag\left(\fT_{\un{m}_1,\fC\otimes\{J_A\}_{\BR}''}^+,\dots,\fT_{\un{m}_{r_1},\fC\otimes\{J_A\}_{\BR}''}^+, \right. \\&\qquad\qquad\qquad \qquad \qquad \left. \fT_{\un{n}_1,\{J_A\}_{\BR}''}^+,\dots,\fT_{\un{n}_{r_2},\{J_A\}_{\BR}''}^+\right)\,
(P\otimes P)^{-1} \nonumber
\end{align*}
and
\begin{align} \label{formula for transpose R}
 \{A^T\}_{\BR}'' \otimes \{A^T\}_\BR''&=\left(P^T \otimes P^T\right)^{-1} \,\left(\{J_A^T\}_{\BR}'' \otimes \{J_A^T\}_{\BR}''\right)\,\left( P^T  \otimes P^T\right)\\&= \left(P^T\otimes P^T\right)^{-1}\,
\diag\left(\fT_{\un{m}_1,\fC\otimes\{J_A^T\}_{\BR}''}^-,\dots,\fT_{\un{m}_{r_1},\fC\otimes\{J_A^T\}_{\BR}''}^-, \right. \nonumber \\ &\left. \qquad \qquad \qquad \qquad \qquad \quad
\fT_{\un{n}_1,\{J_A^T\}_{\BR}''}^-,\dots,\fT_{\un{n}_{r_2},\{J_A^T\}_{\BR}''}^-\right)\,
\left(P^T\otimes P^T\right). \nonumber
\end{align}

 \section{$*$-Linear matrix maps and Hill representations}\label{S:Hill}

 In this section we provide a brief review of $*$-linear matrix maps $\cL$ as in \eqref{cL-Intro} and Hill representations for such maps as studied by R.D. Hill in \cite{H73}, see also \cite{PH81}. For proofs and further details we refer to \cite{tHvdM21b}. Only in the last subsection some new results are presented, on non-minimal Hill representations.

\subsection{$*$-Linear matrix maps}
Consider a linear matrix map $\cL$ as in \eqref{cL-Intro}, that is,
\begin{equation}\label{cL}
\cL:\BF^{q\times q} \to \BF^{n\times n}.
\end{equation}
Following \cite{KMcCSZ19}, we say that $\cL$ is $*$-linear if $\cL(V^*)=\cL(V)^*$ for all $V\in\BF^{q\times q}$. Define the matricization $L$ and Choi matrix $\BL$ associated with $\cL$ via \eqref{Matricization} and \eqref{Choi}, respectively. Then $\cL$ can be written as
\[
\cL(V)= \vect_n^{-1}\left(L\vect_q(V)\right)= \sum_{i,j=1}^q v_{ij}\BL_{ij},\quad V=[v_{ij}]\in\BF^{q\times q}.
\]
$*$-Linearity of $\cL$ can be characterized in the following way.

\begin{theorem}\label{T:*-linear}
Let $\cL$ be a linear map as in \eqref{cL} with matricization $L$ and Choi matrix $\BL$. Decompose $L$ as a block matrix $L=\left[L_{ij}\right]$ where $L_{ij}=\left[\ell^{ij}_{kl}\right]\in\BF^{n \times q}$.  Then the following are equivalent:
\begin{itemize}
  \item[(i)] $\cL$ is $*$-linear;
  \item[(ii)] $\BL\in\cH_{nq}$;
  \item[(iii)] $\ell_{kl}^{ij}=\ov{\ell}_{ij}^{kl}, \quad \text{for all} \quad i,k=1, \ldots,n$ and  $j,l=1,\ldots,q.$
\end{itemize}
\end{theorem}

It is clear what $\BL\in\cH_{nq}$ means. Property (iii) can be interpreted as saying that ``structural properties of $L$ as a block matrix reoccur at the level of the blocks,'' as illustrated by various examples in Subsection 6.1 in \cite{tHvdM21b}. The two observations that are relevant for the present paper are (using the notation of Theorem \ref{T:*-linear}):
\begin{itemize}
\item[(i)] $L_{ij}=0$ if and only if $\ell_{ij}^{kl}=0$ for all $k$ and $l$.

\item[(ii)] $L_{ij}=L_{rs}$ for $(i,j)\neq (r,s)$ if and only if $\ell^{kl}_{ij}=\ell^{kl}_{rs}$ for all $k$ and $l$.

\end{itemize}
This implies in particular that a (upper or lower) block Toeplitz structure in $L$ corresponds to a (upper or lower) Toeplitz structure of the blocks $L_{ij}$.

\subsection{Minimal Hill representations}\label{construction section}

R.D. Hill studied $*$-linear matrix maps $\cL$ as in \eqref{cL} in \cite{H69,H73} via representations of the form
\begin{equation}\label{HillRep}
\cL(V)=\sum_{k,l=1}^m \BH_{kl}\, A_k V A_l^*,\quad V\in\BF^{q \times q},
\end{equation}
for $A_{1},\dots,A_m\in\BF^{n \times q}$ and $\BH=[\BH_{kl}]_{k,l=1}^m\in\BF^{m\times m}$. We call \eqref{HillRep} a {\em Hill representation} of $\cL$ and refer to $\BH$ as the associated {\em Hill matrix}. The Hill representation \eqref{HillRep} of $\cL$ is called minimal if the number $m$ is the smallest among all Hill representations of $\cL$. This number $m$ is equal to the rank of the Choi matrix $\BL$. In the case of a minimal Hill representation, the $A_{1},\dots,A_m\in\BF^{n \times q}$ must be linearly independent.

The results contained in the subsection come from \cite{tHvdM21b}, but most have their origin in \cite{H69,H73}; for proper references please consult \cite{tHvdM21b}.

\begin{theorem}\label{T:Hill}
The linear map $\cL$ in \eqref{cL} is $*$-linear if and only if $\cL$ admits a (minimal) Hill representation with $\BH$ Hermitian. Moreover, $\cL$ is completely positive if and only if $\cL$ admits a (minimal) Hill representation with $\BH$ positive definite.
\end{theorem}

It is possible to express the matricization $L$ and Choi matrix $\BL$ of $\cL$ directly in terms of a minimal Hill representation, as explained in the next result.

\begin{proposition}\label{P:HillLBL}
Let $\cL$ be a $*$-linear map as in \eqref{cL} with a minimal Hill representation \eqref{HillRep}. Then the matricization $L$ and Choi matrix $\BL$ associated with $\cL$ are given by
\begin{equation}\label{LBL-Hill}
L=\sum_{k,l=1}^m \BH_{kl}\, \ov{A}_k\otimes A_l \ands \BL=\widehat{A}^*\BH^T\widehat{A},
\end{equation}
with $\widehat{A}^*:=\begin{bmatrix} \vect_{n \times q}\left(A_1\right) & \hdots & \vect_{n \times q}\left(A_m\right)  \end{bmatrix}\in \BF^{nq \times m}$. Moreover, $\whatA$ has full row rank and $\kr \whatA = \kr \BL$.
\end{proposition}

One of the main new features of \cite{tHvdM21b} is a description of the minimal Hill representations of $\cL$. This works as follows. Write the matricization $L$ of $\cL$ as a block matrix
\begin{equation}\label{Lblock}
L=\left[L_{ij}\right]\quad \mbox{with}\quad L_{ij}\in\BF^{n \times q}\quad\mbox{for $i=1,\ldots,n,\, j=1,\ldots,q$}.
\end{equation}
It then follows that $m=\rank \BL$ satisfies
\begin{equation}\label{mAlt}
m=\dim \tu{span}\{ L_{ij} \colon i=1,\ldots,n, \, j=1,\ldots,q\} \subset \BF^{n\times q}.
\end{equation}
It turns out that the matrices $A_{1},\dots,A_m$ in a minimal Hill representation can be taken such that their span corresponds to the span of the blocks $L_{ij}$ of $L$.

\begin{theorem}\label{T:HillA1Am}
Let $\cL$ be a $*$-linear map as in \eqref{cL} with a minimal Hill representation \eqref{HillRep}. Decompose the matricization $L=\left[L_{ij}\right]$ associated with $\cL$ as in \eqref{Lblock}. Then
\begin{equation}\label{AkCond}
\tu{span}\{ A_{k} \colon k=1,\ldots,m \}=\tu{span}\{ L_{ij} \colon i=1,\ldots,n, \, j=1,\ldots,q\}=:\cW.
\end{equation}
Moreover, for any choice of matrices $A_1,\ldots,A_m\in\BF^{n \times q}$ which satisfy \eqref{AkCond} there exists a matrix $\BH=\left[\BH_{kl}\right]\in\BF^{m \times m}$ so that $\cL$ is given by the corresponding minimal Hill representation \eqref{HillRep}.
\end{theorem}

The criteria \eqref{AkCond} for $A_1,\ldots,A_m$ can also be expressed in terms of the matrix $\whatA$ defined in Proposition \ref{P:HillLBL}.

\begin{proposition}
Assume $\cL$ as in \eqref{cL} is $*$-linear and let $m$ be the rank of the Choi matrix $\BL$ associated with $\cL$. Let $\whatA\in\BF^{m \times nq}$ with $\kr \whatA=\kr \BL$. Then $\whatA$ has full row rank, so that $\whatA\whatA^*$ is invertible, and we have $\BL=\whatA^* \BH^T \whatA$ with
\begin{equation} \label{Choi to Hill}
\BH^T=\left(\whatA\whatA^*\right)^{-1} \whatA \BL \whatA^* \left(\whatA\whatA^*\right)^{-1}.
\end{equation}
In particular, $\cL$ admits a minimal Hill representation \eqref{HillRep} with $A_k=\vect^{-1}_{n\times q} \left(\what{a}_k^T\right)$, $k=1,\ldots,m$, where $\what{a}_k$ is the $k$-th row of $\whatA$.
\end{proposition}

Formula \eqref{LBL-Hill} for the matricization $L$ together with Theorem \ref{T:HillA1Am} yield the following corollary.

\begin{corollary}\label{C:Lincl}
Let $\cL$ be a $*$-linear map as in \eqref{cL} with matricization $L$. Decompose $L$ as in \eqref{Lblock} and define $\cW$ as in \eqref{AkCond}. Then $L\in \ov{\cW} \otimes \cW$.
\end{corollary}

Although any matrices $A_1,\ldots,A_m$ satisfying \eqref{AkCond} can be used in a minimal Hill representation of $\cL$ and $\BH$ can be recovered via \eqref{Choi to Hill}, we present here an explicit construction from \cite{tHvdM21b}, that will also be of use when we consider non-minimal Hill representations. Take $L_1,\ldots,L_m\in\BF^{n \times q}$ which satisfy
\begin{equation}\label{LkCond}
\tu{span}\{ L_{ij} \colon i=1,\ldots,n, \, j=1,\ldots,q \} =\tu{span}\{ L_{k} \colon k=1,\ldots,m \}.
\end{equation}
Then there exists $\al^{ij}_k, \be_{ij}^k\in\BF$ for $i=1,\ldots,n$, $j=1,\ldots,q$ and $k=1,\ldots,m$, so that
\begin{equation}\label{LkLij}
L_k =\sum_{i=1}^n\sum_{j=1}^q \be_{ij}^k L_{ij},\quad L_{ij}=\sum_{k=1}^m \al^{ij}_k L_k.
\end{equation}
Note that the $\al_k^{ij}$ are uniquely determined, but that this is not necessarily the case for the $\be_{ij}^k$.  Define
\begin{equation}\label{AkBkH}
\begin{aligned}
A_k=\left[\ov{\al}_{k}^{ij}\right]\in \BF^{n \times q}, \quad B_k=\left[\be_{ij}^k\right]\in\BF^{n \times q} \quad \text{for} \quad k=1,\ldots,m;\\
\BH=\BH\left(\cL;L_1,\ldots,L_m\right):=\left[\OneVec_n^*\left(B_k \circ \ov{L}_l\right)\OneVec_q\right]_{k,l=1}^m\in\BF^{m \times m}.
\end{aligned}
\end{equation}
Note that we have
\begin{equation}\label{LLk}
L=\sum_{k=1}^m \ov{A}_k \otimes L_k \ands L_k=\left(\OneVec_n \otimes I_n\right)^*\left(\left(B_k \otimes \BBone_{n \times q}\right) \circ L\right)\left(\OneVec_q\otimes I_q\right).
\end{equation}

\begin{theorem}\label{T:HillConstruct}
Let $\cL$ be a $*$-linear map as in \eqref{cL} and select matrices $L_1,\ldots,L_m\in\BF^{n \times q}$ satisfying \eqref{LkCond}. Then $\cL$ is given by the minimal Hill representation \eqref{HillRep} with $A_1,\ldots,A_m$ and $\BH$ as in \eqref{AkBkH}. Moreover, all minimal Hill representations of $\cL$ are obtained in this way.
\end{theorem}

Finally, the minimal Hill representation of a $*$-linear map $\cL$ is unique up to an invertible $m \times m$ matrix. 
See Theorem 5.9 in \cite{tHvdM21b} for the result and an explicit formula for the invertible matrix.

For the Hill matrix $\BH$ as in \eqref{AkBkH} it is always possible to take for $L_1,\ldots,L_m$ $m$ linearly independent matrices among the blocks $L_{ij}$ of L, say
\begin{equation}\label{LkLijChoice}
\begin{aligned}
L_k=L_{i_k j_k},\quad k=1,\ldots,m,&\quad i_k\in\{1,\ldots,n\},\quad j_k \in \{1,\ldots,q\}.
\end{aligned}
\end{equation}
In that case, the Hill matrix can be described in terms of the entries of $L$.

\begin{lemma}\label{L:HillMatSC}
Assume $\cL$ in \eqref{cL} is $*$-linear. Decompose the matricization $L$ as in \eqref{Lblock} with $L_{ij}=\left[\ell_{kl}^{ij}\right]_{k,l=1}^n \in \BF^{n \times q}$ and select linearly independent $L_1,\ldots,L_m$ as in \eqref{LkLijChoice}. Then the Hill matrix $\BH$ in \eqref{AkBkH} determined by $\cL$ and $L_1,\ldots,L_m$ is given by
\[
\BH=\left[\ell_{i_l j_l}^{i_k j_k}\right]_{k,l=1}^m.
\]
\end{lemma}

\subsection{Non-minimal Hill representations}\label{SubS:NonMinHill}

For the purpose of this paper it will also be convenient to consider non-minimal Hill representations. In this subsection we look at what parts of the construction of minimal Hill representations at end of the previous subsection still works in the non-minimal case.

Choose $L_1, \ldots, L_r \in \BF^{n \times q}$ among the blocks $L_{ij}$ of L:
\begin{equation}\label{LkLijChoice non-min}
\begin{aligned}
L_k=L_{i_k j_k},\quad k=1,\ldots,r,&\quad i_k\in\{1,\ldots,n\}, \quad j_k \in \{1,\ldots,q\},
\end{aligned}
\end{equation}
where  $r\ge m=\rank \,\BL$, so that
\begin{equation}\label{span prop}
\textup{span}\{L_1, \ldots, L_r\}=\textup{span}\{L_{ij}: i=1,\ldots,n, \, j=1,\ldots,q\}.
\end{equation}
Following the construction as given in Subsection \ref{construction section}, there exist scalars $\al^{ij}_k, \be_{ij}^k\in\BF$ for $i=1,\ldots,n$, $j=1,\ldots,q$ and $k=1,\ldots,r$, so that
\begin{equation}\label{LkLij non-min}
L_k =\sum_{i=1}^n\sum_{j=1}^q \be_{ij}^k L_{ij},\quad L_{ij}=\sum_{k=1}^r \al^{ij}_k L_k.
\end{equation}
In this case, both the $\be_{ij}^k$ and the $\al^{ij}_k$ need not be unique. It is, in particular, possible to choose these numbers so that
\begin{align}\label{choice for B_k}
    \beta_{i_kj_k}^k&=1 \quad \text{and} \quad \beta_{ij}^k=0 \quad \text{for all} \quad (i,j)\ne(i_k,j_k);\\ \label{choice for A_k} \alpha_k^{i_kj_k}&=1  \quad \text{and} \quad \alpha_{l}^{i_kj_k}=0 \quad \text{for all} \quad l\ne k,
\end{align}
and in the remainder of this section we shall assume \eqref{choice for B_k} and \eqref{choice for A_k} to hold. Set
\begin{equation}\label{AkBk non-min}
A_k=\left[\ov{\al}_{k}^{ij}\right]\in \BF^{n \times q} \quad \text{and} \quad B_k=\left[\be_{ij}^k\right]\in\BF^{n \times q} \quad \text{for} \quad k=1,\ldots,r.
\end{equation}
From our choice \eqref{choice for A_k} we see that the matrices $A_1, \ldots, A_r$ are still linearly independent since the $(i_k,j_k)$-th entry of $\sum_{k=1}^r \eta_k A_k$ is equal to $\eta_k$. Hence $\sum_{k=1}^r \eta_k A_k=0$ holds only in the case where $\eta_k=0$ for all $k$. Also note that \eqref{choice for B_k} gives $B_k=\mathcal{E}^{(n,q)}_{i_kj_k}$ for all $k.$ It still follows that
\begin{equation}\label{LLk non-min}
L=\sum_{k=1}^r \ov{A}_k \otimes L_k \ands L_k=\left(\OneVec_n \otimes I_n\right)^*\left(\left(B_k \otimes \BBone_{n \times q}\right) \circ L\right)\left(\OneVec_q\otimes I_q\right).
\end{equation}
Define the Hill matrix associated with the selection $L_1,\ldots,L_r$ as:
\begin{equation}\label{Hill non-min}
\widetilde{\BH}=\widetilde{\BH}\left(\cL;L_1,\ldots,L_r\right):=\left[\OneVec_n^*\left(B_k \circ \ov{L}_l\right)\OneVec_q\right]_{k,l=1}^r\in\BF^{r \times r}.
\end{equation}

Then the $*$-linear map $\cL$ admits a non-minimal Hill representation with $A_1, \ldots, A_r$ and $\widetilde{\BH}$ as constructed above, as follows from the next theorem, which is an analogy to our main result in \cite[Theorem 5.1]{tHvdM21b}.

\begin{theorem}
Assume $\cL$ as in \eqref{cL} is $*$-linear with matricization $L$ and Choi matrix $\BL$. Choose $L_1, \ldots, L_r$ so that \eqref{LkLijChoice non-min} and \eqref{span prop} hold and define $A_1, \ldots, A_r$ and $\widetilde{\BH}$ as above. Then \begin{equation}\label{LBL-Hill non min}
\cL(V)=\sum_{k,l=1}^r \widetilde{\BH}_{kl}A_lVA_k^*, \quad
L=\sum_{k,l=1}^r \widetilde{\BH}_{kl}\, \ov{A}_k\otimes A_l \ands \BL=\widehat{A}^*\widetilde{\BH}^T\widehat{A},
\end{equation}
for each $V \in \BF^{q \times q}$, where $\widehat{A}^*:=\begin{bmatrix} \vect_{n \times q}\left(A_1\right) & \hdots & \vect_{n \times q}\left(A_r\right)  \end{bmatrix}\in \BF^{nq \times r}$. Moreover, $\widetilde{\BH}$ is in $\cH_r$ with $\rank \, \widetilde{\BH}=\rank \,\BL$ and $\whatA$ has full row rank.
\end{theorem}

\begin{proof}[\bf Outline of proof]
From the choice we made for $A_k$ in \eqref{choice for A_k} and $B_k$ in \eqref{choice for B_k}, for $k=1,\ldots,r$, the identities in Lemma 5.2 in \cite{tHvdM21b}, given by
\[\OneVec_n^*\left(B_k \circ \ov{A}_k\right)\OneVec_q=\alpha_k^{i_kj_k}=1 \quad \text{and} \quad \OneVec_n^*\left(B_k \circ \ov{A}_l\right)\OneVec_q=\alpha_l^{i_kj_k}=0 \quad \text{for} \quad l \neq k,\] still hold. Since only these identities are needed to prove Proposition 5.3 in \cite{tHvdM21b}, we get the representation of $L$ in \eqref{LBL-Hill non min} and the fact that $\widetilde{\BH}$ belong to $\cH_r$. The representation of $\BL$ and the non-minimal Hill representation of $\cL$ in \eqref{LBL-Hill non min} follows by the same arguments as in the proof of Theorem 5.1 in \cite{tHvdM21b}. The full row rank of $\widehat{A}$ follows by Lemma \ref{L:MatTrans}. Now, since $\BL=\widehat{A}^*\widetilde{\BH}^T\widehat{A}$, with $\widehat{A}$ having full row rank, it follows at once from \eqref{LBL-Hill non min} that $\rank\, \BL = \rank \, \widetilde{\BH}.$
\end{proof}

Hence, $\widetilde{\BH}$ need not be invertible, as in case of a minimal Hill representation. Moreover, we no longer have that $\kr \,\BL=\kr \widehat{A}$. This is because $\rank \widehat{A}=r$ and $\rank \, \BL = m$, and therefore $\dim \kr \, \BL=nq-m$ and $\dim \kr \widehat{A}=nq-r.$ When $r>m$,  we only have the one inclusion $\kr \widehat{A} \subset \kr \BL$, since the dimensions do not add up. Furthermore, from the fact that \[\dim \textup{span} \{A_1, \ldots, A_r\} =r \ands \dim \textup{span} \{L_{ij}: \, i=1, \ldots, n, \, j=1,\ldots,q\}=m\] it follows that the equality in \eqref{AkCond}, with $m$ replaced by $r$, no longer hold. Instead, from (iii) in Proposition 5.3 in \cite{tHvdM21b} we get \[\textup{span} \{L_{ij}: \, i=1, \ldots, n, \, j=1,\ldots,q\}=\textup{span}\{L_1, \ldots, L_r\}\subset \textup{span} \{A_1, \ldots,A_r\}.\]

We conclude with a result on the complete positivity of a $*$-linear map $\cL$ as constructed above.
\begin{theorem}\label{T:Hill non-min}
Assume $\cL$ as in \eqref{cL} is $*$-linear. Define $L$ as in \eqref{Matricization} and $\BL$ as in \eqref{Choi} and let $m=\rank \, \BL.$ Choose $L_1, \ldots, L_r$ so that \eqref{LkLijChoice non-min} and \eqref{span prop} hold and define $A_1, \ldots, A_r$ as in \eqref{AkBk non-min} and $\widetilde{\BH}$ as in \eqref{Hill non-min}. Then $\cL$ is completely positive if and only if $\widetilde{\BH}$ is positive semidefinite.
\end{theorem}

\begin{proof}[\bf Proof]
Since complete positivity of $\cL$ coincides with the positive semidefiniteness of the Choi matrix $\BL$, from the representation of $\BL$ in \eqref{LBL-Hill non min} along with the fact that $\widehat{A}$ has full row rank, the result follows.
\end{proof}

Hence, to verify complete positivity with a non-minimal Hill representation as constructed above, in particular, with \eqref{choice for A_k} and \eqref{choice for B_k}, checking whether the Hill matrix is positive semidefinite is still both a necessary and sufficient condition. With other choices of non-minimal Hill representations, this need not be the case.

\section{Positive maps that are also completely positive} \label{pos imply completely pos}

We start this section with an observation made in \cite{tHvdM21c} which enabled us to determine various classes of $*$-linear matrix maps $\cL$ for which positivity and complete positivity coincide. Recall that $\cL$ is completely positive if and only if the Choi matrix $\BL$ of $\cL$ is positive semidefinite. Whether $\cL$ is positive is not so easy to determine. The following proposition provides a necessary and sufficient criteria in terms of the Choi matrix. The result is essentially contained in Propositions 3.1 and 3.6 of \cite{KMcCSZ19}.

\begin{proposition}\label{P:cL-Pos}
A $*$-linear map $\cL$ as in \eqref{cL} is positive if and only if the Choi matrix $\BL$ in \eqref{Choi} satisfies
\begin{equation}\label{PosCon2}
\left(z\otimes x\right)^* \BL \left(z\otimes x\right) \ge 0 \quad \mbox{for all}\quad x\in\BF^n \ands z \in \BF^q.
\end{equation}
\end{proposition}

Assume a $*$-linear map $\cL$ in \eqref{cL} is given by a minimal Hill representation \eqref{HillRep} with Hill matrix $\BH$ and Choi matrix $\BL$. Define $\widehat{A}$ as in Proposition \ref{P:HillLBL}. From
\[
\left(z\otimes x\right)^* \BL \left(z\otimes x\right)=\left(z\otimes x\right)^* \widehat{A}^*\BH^T \widehat{A} \left(z\otimes x\right)   \quad \mbox{for all}\quad x\in\BF^n \ands z \in \BF^q,
\]
we see that positivity of $\cL$ is equivalent to $y^* \BH^T y\geq 0$ for all $y$ from the set
\begin{equation}\label{y-set}
\fY_{\whatA}:=
\{ \widehat{A}(z \otimes x) \colon x \in \BF^{n}, \, z \in \BF^q \}.
\end{equation}
Complete positivity in turn, by Theorem \ref{T:Hill}, is equivalent to $\BH \geq 0$, which is the same as  $\BH^T \geq 0$. In particular, it follows that positivity and complete positivity of $\cL$ coincide when $\fY_{\whatA}=\BF^m$.

Note that
\[
\widehat{A}(z \otimes x) = \left(\widehat{A}(I_q \otimes x)\right)z = \left(\widehat{A}(z \otimes I_n)\right)x.
\]
From this it is clear that $\fY_{\widehat{A}}=\BF^m$ holds whenever we can find a vector $x \in \BF^n$ such that the matrix $\widehat{A}(I_q \otimes x) \in \BF^{m \times q}$ has full row rank or a vector $z \in \BF^q$ such that the matrix $\widehat{A}(z \otimes I_n) \in \BF^{m \times n}$ has full row rank. Clearly there exists no such vector $x \in \BF^n$ if $m >q$ and no such vector $z \in \BF^q$ is $m >n$. Hence proving that $\fY_{\widehat{A}}=\BF^m$ in this way can only be done if $\rank \BL$ is at most $\max\{n,q\}$. We further point out that whether $\widehat{A}(I_q \otimes x)$ or $\widehat{A}(z \otimes I_n)$ has full row rank is independent of the choice of the minimal Hill representation, since the matrix $\whatA$ is unique up to multiplication on the left by a $T\in \textup{GL}(m,\BF)$; see Theorem 5.9 in \cite{tHvdM21b}. Hence, these are properties of the $*$-linear map $\cL$ and not of a specific minimal Hill representation.

\begin{theorem} \label{property C1}
Let $\cL$ in \eqref{cL} be a $*$-linear map with matricization $L$ and Choi matrix $\BL$. Set $m=\rank \BL$, decompose $L$ as in \eqref{Lblock} and define $\cW$ as in \eqref{AkCond}. Then for any minimal Hill representation \eqref{HillRep}, for $\whatA$ defined as in Proposition \ref{P:HillLBL} there exists a vector $z\in \BF^q$ such that $\widehat{A}(z \otimes I_n)$ has full row-rank if and only if the subspace $\cW$ has the following property:
\begin{align*}
\textup{(C1)} \quad &\text{For any linearly independent } X_1, \ldots, X_k \text{ in } \cW, \text{ there exists a } v \in \BF^q \\ & \qquad \qquad \qquad \qquad \text{ such that } X_1v, \ldots, X_kv \text{ is linearly independent in } \BF^n.
\end{align*}
Hence, if \textup{(C1)} holds, then $\fY_{\widehat{A}}=\BF^m$  and positivity and complete positivity of $\cL$ coincide. Moreover, in that case $\dim \cW\leq n$.
\end{theorem}

\begin{proof}[\bf Proof]
We already observed above that the choice of the minimal Hill representation is irrelevant, by \cite[Theorem 5.9]{tHvdM21b}, so we may select a minimal Hill representation \eqref{HillRep} of $\cL$ arbitrarily. Then it remains to show that there exists a vector $z\in \BF^q$ such that $\widehat{A}(z \otimes I_n)$ has full row-rank if and only if (C1) holds.

First assume that (C1) holds. We know from Theorem \ref{T:HillA1Am} that $A_1,\ldots,A_m$ are linearly independent matrices contained in $\cW$. Hence there exists a vector $v\in\BF^q$ so that $A_1v,\ldots,A_mv$ are linearly independent in $\BF^n$. Equivalently, by taking adjoints, the vectors $\ov{v}^T A_1^*,\ldots,\ov{v}^T A_m^*$ are linearly independent in $\BF^{1 \times n}$ which is the same as saying that $(I_m\otimes \ov{v})^T\widecheck{A}$ has full row-rank, with $\widecheck{A}$ defined in the same way as $\widecheck{K}$ in \eqref{MatTrans}. By the first identity in \eqref{KtilKhatRel} in Lemma \ref{L:MatTrans}, $(I_m\otimes \ov{v})^T\widecheck{A}$  having full row rank is the same as $\widehat{A}(\ov{v} \otimes I_n)$ having full row rank, and thus our claim follows.

Conversely, assume there exists a $z\in \BF^q$ such that $\widehat{A}(z \otimes I_n)$ has full row-rank. Let $X_1,\ldots,X_k$ be linearly independent matrices in $\cW$. Then $k\leq m=\rank \BL=\dim \cW$. Without loss of generality we can assume $k=m$, extending $X_1,\ldots,X_k$ to a basis of $\cW$ if needed. Again using the fact that $A_1,\ldots,A_m$ can be chosen arbitrarily, as long as \eqref{AkCond} holds, we may assume that $A_j=X_j$ for $j=1,\ldots,m$. Then, since $\widehat{A}(\ov{v} \otimes I_n)$ for $v=\ov{z}$ has full row-rank, by the analysis given above, which goes two ways, it follows that $A_1v,\ldots,A_mv$ are linearly independent, and hence $X_1v,\ldots,X_kv$ are linearly independent.
\end{proof}

The following result characterizes when there exists a vector $x\in\BF^n$ so that $\widehat{A}(I_q \otimes x)$ has full row-rank. The proof is analogous to the proof of Theorem \ref{property C1}, now relying on the second identity in \eqref{KtilKhatRel} rather than the first, so it will be omitted.

\begin{theorem} \label{property C2}
Let $\cL$ in \eqref{cL} be a $*$-linear map with matricization $L$ and Choi matrix $\BL$. Set $m=\rank \BL$, decompose $L$ as in \eqref{Lblock} and define $\cW$ as in \eqref{AkCond}. Then for any minimal Hill representation \eqref{HillRep}, for $\whatA$ defined as in Proposition \ref{P:HillLBL} there exists a vector $x\in \BF^n$ such that $\widehat{A}(I_q \otimes x)$ has full row-rank if and only if the subspace $\cW$ has the following property:
\begin{align*}
\textup{(C2)} \quad &\text{For any linearly independent } X_1, \ldots, X_k \text{ in } \cW, \text{ there exists a } v \in \BF^n \\ &\qquad \qquad \qquad \qquad \text{ such that }  X_1^*v, \ldots, X_k^*v \text{ is linearly independent in } \BF^q.
\end{align*}
Hence, if \textup{(C2)} holds, then $\fY_{\widehat{A}}=\BF^m$  and positivity and complete positivity of $\cL$ coincide. Moreover, in that case $\dim \cW\leq q$.
\end{theorem}

Although $\cW$ in \eqref{AkCond} is the minimal subspace one can work with, sometimes it is convenient to identify a larger subspace $\cV$ so that $L \in \ov{\cV}\otimes \cV$, since $\cV$ may better capture the structural properties of $L$. In this case it is still sufficient that $\cV$ satisfies (C1) or (C2) to conclude that $\fY_{\widehat{A}}=\BF^m$, but possibly not necessary. This claim will be proved in Theorem \ref{T:Lin} below.

\subsection{Subspaces of $\BF^{n \times q}$ satisfying condition (C1) or (C2)}

Since subspaces $\cW$ of $\BF^{n \times q}$ satisfying (C1) or (C2) are of particular importance to us, we further investigate such subspaces in this subsection. We start with a duality result between (C1) and (C2).

\begin{lemma} \label{C1C2 duality}
A subspace $\cW$ of $\BF^{n \times q}$ satisfies condition \tu{(C1)} if and only if $\cW^*$ satisfies \tu{(C2)}.
\end{lemma}

\begin{proof}[\bf Proof]
Assume $\cW$ satisfies property (C1). Let $Y_1,\ldots,Y_k\in \cW^*$ be linearly independent. Then $Y_1^*,\ldots, Y_k^*$ are in $\cW$ and are also linearly independent. Hence there exists a vector $v\in \BF^q$ so that $Y_1^*v,\ldots, Y_k^*v$ are linearly independent in $\BF^n$. This shows that $\cW^*$ satisfies (C2). The converse implication follows by a similar argument.
\end{proof}

Because of this duality relation between properties (C1) and (C2) we shall prove our results only for property (C1) and list the analogous results for property (C2), in Lemma \ref{L: (C2) props} below.

\begin{lemma}\label{complex con C1}
A subspace $\cW$ of $\BF^{n \times q}$ satisfies condition \textup{(C1)} if and only if $\ov{\cW}$ satisfies \tu{(C1)}.
\end{lemma}

\begin{proof}[\bf Proof]
Since linearly independence of $X_1, \ldots, X_k$ is equivalent to $\ov{X}_1, \ldots, \ov{X}_k$ being linearly independent, and the same is true for $X_1v, \ldots, X_kv$ and $\ov{X}_1\ov{v}, \ldots, \ov{X}_k\ov{v}$, the result follows.
\end{proof}

\begin{lemma} \label{similar alg}
Suppose $\cW$ is a subspace of $\BF^{n \times q}$ that satisfies the condition \tu{(C1)}. Then $P \cW Q$ also satisfies \tu{(C1)} for any $P \in \textup{GL}(n,\BF)$ and $Q \in \textup{GL}(q,\BF)$.
\end{lemma}

\begin{proof}[\bf Proof]
Take any linearly independent $X_1, \ldots, X_k$ in $P\cW Q$. Then $X_i=PY_iQ$, where $Y_1, \ldots, Y_k$ are linearly independent matrices in $\cW$. Hence there exists a vector $v \in \BF^q$ such that $Y_1v, \ldots, Y_kv$ are linearly independent vectors in $\BF^n$. Define $u:=Q^{-1}v \in \BF^q$ and note that for all $\al_1,\ldots,\al_k\in\BF$ we have
\[
P\left(\sum_{i=1}^k\alpha_iY_iv\right)=\sum_{i=1}^k\alpha_iPY_iv=\sum_{i=1}^k\alpha_iX_iQ^{-1}v=\sum_{i=1}^k \alpha_iX_iu.
\]
The invertibility of $P$ together with the linear independence of $Y_1v, \ldots, Y_kv$ implies that $X_1u, \ldots, X_ku$ are linearly independent. Hence $P \cW Q$ satisfies (C1).
\end{proof}

\begin{lemma} \label{sub alg}
Suppose $\cW$ is subspace of $\BF^{n \times q}$ that satisfies condition \tu{(C1)}. Then any subspace $\cV$ of $\cW$ also satisfies \tu{(C1)}.
\end{lemma}

\begin{proof}[\bf Proof]
The claim follows because any linearly independent $X_1, \ldots, X_k$ in $\cV$ are also linearly independent matrices in $\cW$.
\end{proof}

\begin{lemma}\label{direct sum of alg}
Suppose $\cW_1$ is a subspace of $\BF^{n_1 \times q_1}$ and $\cW_2$ is a subspace of $\BF^{n_2 \times q_2}$, that both satisfy the condition \tu{(C1)}. Then $\cW_1 \oplus \cW_2$ in $\BF^{n_1+n_2 \times q_1+q_2}$ also satisfies the condition \tu{(C1)}.
\end{lemma}

\begin{proof}[\bf Proof]
Take any linearly independent $X_1, \ldots, X_k$ in $\cW_1 \oplus \cW_2$, where $X_i=Y_i\oplus Z_i$ with $Y_i \in \cW_1$ and $Z_i \in \cW_2$ for $i=1,\ldots,k$. Write $k_1+k_2=k$ with $k_1=\textup{dim span}\{Y_1, \ldots, Y_k\}$. Reorder the set $X_1,\ldots,X_k$ so that $Y_1, \ldots, Y_{k_1}$ are linearly independent in $\cW_1$ and hence we can write $Y_i=\sum_{j=1}^{k_1} \alpha_j^{i}Y_j$ with $\alpha_{j}^i \in \mathbb{F}$ for all $i$. Let $n:=n_1+n_2$, $q:=q_1+q_2$ and \[ X:=\begin{bmatrix} X_1 & \hdots & X_k\end{bmatrix} \in \BF^{n\times kq}.\] 
Decompose $X$ as \begin{align*}
    X&=\begin{bmatrix}
    Y & \wtil{Y} \\ Z & \wtil{Z}
    \end{bmatrix} \quad \text{where} \quad Y:= \begin{bmatrix}
    Y_1 & 0_{n_1 \times q_2} & \hdots & Y_{k_1} & 0_{n_1 \times q_2}
    \end{bmatrix} \in \BF^{n_1 \times k_1q}, \\ \wtil{Y}&:= \begin{bmatrix}
    Y_{k_1+1} & 0_{n_1 \times q_2} & \hdots & Y_{k} & 0_{n_1 \times q_2}
    \end{bmatrix} \in \BF^{n_1 \times k_2q}, \\ Z&:= \begin{bmatrix}
    0_{n_2 \times q_1} & Z_1  & \hdots & 0_{n_2 \times q_1} & Z_{k_1}
    \end{bmatrix} \in \BF^{n_2 \times k_1q} \ands \\ \wtil{Z}&:= \begin{bmatrix}
    0_{n_2 \times q_1} & Z_{k_1+1} & \hdots & 0_{n_2 \times q_1} & Z_k
    \end{bmatrix} \in \BF^{n_2 \times k_2q}.
\end{align*} Define \begin{align*}
    T&:= -\begin{bmatrix}
    \alpha_1^{k_1+1} & \hdots & \alpha_1^k  \\ \vdots &   \ddots & \vdots  \\ \alpha_{k_1}^{k_1+1}  & \hdots & \alpha_{k_1}^k
    \end{bmatrix} \in \BF^{k_1 \times k_2}, \quad \widehat{U}:=\begin{bmatrix}
    I_{k_1} & T \\ 0_{k_2 \times k_1} & I_{k_2}
    \end{bmatrix} \in \BF^{k \times k}
\end{align*}  and \[U:=\left(\widehat{U} \otimes I_q\right) \in \BF^{kq \times kq}.\]
Then \begin{align*}
    XU&=\begin{bmatrix}
    Y & \wtil{Y} \\ Z & \wtil{Z}
    \end{bmatrix}\begin{bmatrix}
    I_{k_1q} & T \otimes I_q \\ 0_{k_2q \times k_1q} & I_{k_2q}
    \end{bmatrix} = \begin{bmatrix}
    Y & Y\left(T \otimes I_q\right) + \wtil{Y} \\ Z & Z\left(T \otimes I_q\right)+\wtil{Z}
    \end{bmatrix}=\begin{bmatrix}
    Y & 0_{n_1 \times k_2q} \\ Z & \what{Z}
    \end{bmatrix},
\end{align*}  where  \[\widehat{Z}:=Z\left(T \otimes I_q\right)+\widetilde{Z}=\begin{bmatrix}0_{n_2 \times q_1} & \widehat{Z}_1 & \hdots & 0_{n_2 \times q_1}& \widehat{Z}_{k_2} \end{bmatrix}\] with $\widehat{Z}_l \in \cW_2$ for $l=1, \ldots, k_2.$ Indeed, the $0$ entry in the right upper corner follows since \begin{align*} -Y\left(T\otimes I_q\right)=Y\left(-T \otimes I_q\right)&=\begin{bmatrix}
    Y_1 & 0_{n_1 \times q_2} & \hdots & Y_{k_1} & 0_{n_1 \times q_2}
    \end{bmatrix} \\ &\qquad \times \begin{bmatrix}
    \alpha_1^{k_1+1}I_{q_1} & 0_{q_1 \times q_2} & \hdots & \alpha_1^kI_{q_1} & 0_{q_1 \times q_2} \\ 0_{q_2 \times q_1} &\alpha_1^{k_1+1}I_{q_2}  & \hdots & 0_{q_2 \times q_1} & \alpha_1^{k}I_{q_2}  \\ \vdots & \vdots & \ddots & \vdots & \vdots \\ \alpha_{k_1}^{k_1+1}I_{q_1} & 0_{q_1 \times q_2} & \hdots & \alpha_{k_1}^kI_{q_1} & 0_{q_1 \times q_2}\\0_{q_2 \times q_1} & \alpha_{k_1}^{k_1+1}I_{q_2} & \hdots & 0_{q_2 \times q_1} & \alpha_{k_1}^{k}I_{q_2}
    \end{bmatrix} \\ &=\begin{bmatrix}
    \sum_{j=1}^{k_1}\alpha_j^{k_1+1}Y_j & 0_{n_1 \times q_2} & \hdots & \sum_{j=1}^{k_1}\alpha_{j}^{k}Y_{j} & 0_{n_1 \times q_2}
    \end{bmatrix} \\ &= \begin{bmatrix}
    Y_{k_1+1} & 0_{n_1 \times q_2} & \hdots & Y_{k} & 0_{n_1 \times q_2}
    \end{bmatrix} =\wtil{Y}.\end{align*} Let $\widetilde{X}:=XU,$ with $\widetilde{X}$ decomposed as $\widetilde{X}=\begin{bmatrix}\widetilde{X}_1 & \hdots & \widetilde{X}_k \end{bmatrix}$. Since $U$ is invertible and $X_1, \ldots, X_k$ are linearly independent, it follows that $\widetilde{X}_1, \ldots, \widetilde{X}_k$ are linearly independent. This can only be the case if $\widehat{Z}_1, \ldots, \widehat{Z}_{k_2}$ are linearly independent in $\cW_2$.  
 Now we make use of the fact that there exists a vector $v \in \BF^{q_1}$ and $w \in \BF^{q_2}$ such that $Y_1v,\ldots,Y_{k_1}v$ form a linearly independent set in $\BF^{n_1}$ and $\widehat{Z}_{1}w,\ldots,\widehat{Z}_{k_2}w$ form a linearly independent set in $\BF^{n_2}.$ Let $\widetilde{u}:=\begin{bmatrix} v^T &  w^T\end{bmatrix}^T \in \BF^{q}$. Then \begin{align*}\widetilde{X}\left(I_{k} \otimes \widetilde{u}\right) &=\begin{bmatrix}
 Y & 0_{n_1 \times k_2q} \\ Z & \what{Z}
 \end{bmatrix}\left(I_k \otimes \widetilde{u}\right)=\begin{bmatrix}
 Y & 0_{n_1 \times k_2q} \\ Z & \what{Z}
    \end{bmatrix}\begin{bmatrix}I_{k_1} \otimes \widetilde{u} & 0_{k_1q \times k_2} \\ 0_{k_2q \times k_1} & I_{k_2} \otimes \widetilde{u} \end{bmatrix}\\&=\begin{bmatrix}
    Y\left(I_{k_1} \otimes \widetilde{u}\right) & 0_{n_1 \times k_2}\\ Z\left(I_{k_1} \otimes \widetilde{u}\right) & \what{Z}\left(I_{k_2} \otimes \widetilde{u}\right)
    \end{bmatrix} \\ &= \begin{bmatrix} \begin{bmatrix} Y_1v & \hdots & Y_{k_1}v \end{bmatrix} & 0_{n_1 \times k_2} \\ \begin{bmatrix} Z_1w & \hdots & Z_{k_1}w \end{bmatrix} & \begin{bmatrix} \widehat{Z}_1w & \hdots & \widehat{Z}_{k_2}w \end{bmatrix} \end{bmatrix}.\end{align*} Hence there exists a vector $\widetilde{u} \in \BF^q$ such that \[\rank \widetilde{X}\left(I_k \otimes \widetilde{u}\right) = k,\] since \begin{align*}\rank \widetilde{X}\left(I_k \otimes \widetilde{u}\right)&=\rank \begin{bmatrix} Y_1v & \hdots & Y_{k_1}v \end{bmatrix}+\rank\begin{bmatrix} \widehat{Z}_1w & \hdots & \widehat{Z}_{k_2}w \end{bmatrix} \\&=k_1+k_2=k.\end{align*}
    Now from \[\widetilde{X}\left(I_k \otimes \widetilde{u}\right)=XU\left(I_k \otimes \widetilde{u}\right)=X\left(\widehat{U} \otimes I_q\right)\left(I_k \otimes \widetilde{u}\right)=X\left(\widehat{U} \otimes \widetilde{u}\right)=X\left(I_k \otimes \widetilde{u}\right)\widehat{U},\] and the fact that $\widehat{U}$ is invertible it follows that
    \[\rank X\left(I_k \otimes \widetilde{u}\right)=\rank \begin{bmatrix}X_1\widetilde{u} & \hdots & X_k\widetilde{u} \end{bmatrix} =k.
\]
This proves that there exists a vector $\widetilde{u}$ such that $X_1\widetilde{u}, \ldots, X_k\widetilde{u}$ are linearly independent in $\BF^n$. Hence $\cW_1 \oplus \cW_2$ satisfy $(C1)$ whenever both $\cW_1$ and $\cW_2$ satisfy $(C1).$
\end{proof}

Via the duality relation between properties (C1) and (C2) obtained in Lemma \ref{C1C2 duality}, the following result follows directly from Lemmas \ref{similar alg} -- \ref{direct sum of alg}.

\begin{lemma}\label{L: (C2) props}
Let $\cW$ be a subspace of $\BF^{n \times q}$ that satisfies the condition \tu{(C2)}. Then:
\begin{itemize}
\item[(i)] $\ov{\cW}$ also satisfies \tu{(C2)};

\item[(ii)] $P \cW Q$ also satisfies \tu{(C2)} for any $P \in \textup{GL}(n,\BF)$ and $Q \in \textup{GL}(q,\BF)$;

\item[(iii)] any subspace $\cV$ of $\cW$ satisfies \tu{(C2)};

\item[(iv)] if $\cZ$ is a subspace of $\BF^{n' \times q'}$ which satisfies \tu{(C2)}, then $\cW\oplus \cZ$ is a subspace in $\BF^{(n+ n') \times (q+q')}$ that satisfies \tu{(C2)}.

\end{itemize}
\end{lemma}

Combining the results of Lemmas \ref{C1C2 duality} and \ref{complex con C1} gives the another duality between conditions (C1) and (C2).

\begin{corollary}\label{duality transpose}
The subspace $\cW$ of $\BF^{n \times q}$ satisfies condition $(C1)$ if and only if $\cW^T$ satisfies $(C2)$.
\end{corollary}

We can now prove the following addition to Theorems \ref{property C1} and \ref{property C2}.

\begin{theorem}\label{T:Lin}
Let $\cL$ in \eqref{cL} be a $*$-linear map with matricization $L$ and Choi matrix $\BL$. Let $\cV\subset \BF^{n\times q}$ be a subspace so that $L\in \ov{\cV} \otimes \cV$. Assume that $\cV$ satisfies $(C1)$ (resp.\ $(C2)$). Then for any minimal Hill representation \eqref{HillRep}, for $\whatA$ defined as in Proposition \ref{P:HillLBL} there exists a vector $z\in \BF^q$ such that $\widehat{A}(z\otimes I_n)$ has full row-rank (resp.\ a vector $x\in \BF^n$ such that $\widehat{A}(I_q \otimes x)$ has full row-rank).
\end{theorem}

\begin{proof}[\bf Proof]
The fact that $L\in \ov{\cV} \otimes \cV$ implies that the blocks $L_{ij}$ of $L$ are all contained in $\cV$. Then also $\cW$ defined in \eqref{AkCond} is contained in $\cV$. By Lemma \ref{sub alg} and item (iii) in Lemma \ref{L: (C2) props} it follows that $\cW$ has property (C1) or property (C2) whenever $\cV$ has property (C1) or property (C2), respectively. The claims of the theorem then follows from Theorems \ref{property C1} and \ref{property C2}.
\end{proof}

So far we have not seen examples of subspaces $\cW$ satisfying (C1) or (C2). For the purpose of the present paper we are only interested in a specific class of subspaces for which this occurs, which will be proved in the next subsection. The problem to characterize subspaces $\cW$ satisfying (C1) or (C2) is left as a question.

\begin{question}
Which subspaces $\cW$ of $\BF^{n \times q}$ satisfy (C1) and which satisfy (C2)?
\end{question}

\subsection{The bicommutant of a matrix satisfies (C1) and (C2)}\label{SubS:bicomm}

Our main result in this subsection is the following theorem.

\begin{theorem}\label{main result for double commutant}
For any matrix $A \in \BF^{n \times n}$  the algebra $\left\{A\right\}_{\BF}''$ satisfies both conditions \tu{(C1)} and \tu{(C2)}.
\end{theorem}

To prove this result we first prove the claim for some special case.

\begin{lemma} \label{alg lower triangular toeplitz}
The matrix algebras $\fT_{n,\BF}^{-}$ and $\fT_{n,\BF}^{+}$ both satisfy conditions \tu{(C1)} and \tu{(C2)}.
\end{lemma}

\begin{proof}[\bf Proof]
We prove the $\fT_{n,\BF}^{-}$ satisfies (C1) by showing we can always take $v=e_1$. Note that $\fT_{n,\BF}^{-}=\{S_n^T\}''_{\BF}$ is commutative. Take any linearly independent $X_1, \ldots, X_k$ in $\fT_{n,\BF}^{-}$. Suppose
\begin{align*}
\sum_{i=1}^k \alpha_iX_ie_1=0, \quad \text{where} \quad \alpha_i \in \BF \quad \text{with} \quad i=1, \ldots,k.
\end{align*}
Then for $j=0, \ldots, n-1$ we have
\[
0=\left(S_n^T\right)^j\sum_{i=1}^k \alpha_iX_ie_1=\sum_{i=1}^k \alpha_i\left(S_n^T\right)^jX_ie_1=\sum_{i=1}^k \alpha_iX_i\left(S_n^T\right)^je_1=\sum_{i=1}^k \alpha_iX_ie_{j+1}
\]
Hence all columns of the matrix $\sum_{i=1}^k \alpha_iX_i$ are zero, so that $\sum_{i=1}^k \alpha_iX_i=0$. The linear independence of $X_1, \ldots, X_k$ then implies that $\al_1=\cdots=\al_k=0$. Hence $X_1e_1, \ldots, X_k e_1$ are linearly independent, and thus $\fT_{n,\BF}^{-}$ satisfies (C1).


The proof of (C2) goes analogously, with the required vector being  $e_n \in \mathbb{F}$. The claims for $\fT^+_{n,\BF}$ follow by the duality result of Lemma \ref{C1C2 duality} since $(\fT^-_{n,\BF})^*=\fT^+_{n,\BF}$.
\end{proof}

\begin{lemma} \label{alg lower triangular toeplitz real}
The matrix algebras $\fT_{n,\fC}^{-}$ and $\fT_{n,\fC}^{+}$ both satisfy conditions \tu{(C1)} and \tu{(C2)}.
\end{lemma} \begin{proof}[\bf Proof]
For the proof of (C1) we make use of the fact that 
$\fT_{n,\fC}^{-}$ is a commutative algebra. By Lemma \ref{result on tensors and algebras} we know $\fT_{n,\fC}^{-}=\fT_{n,\BF}^{-}\otimes \fC$ and hence $\left(\left(S_n^T\right)^l\otimes I_2\right) \in \fT_{n,\fC}^{-}$ for all $l=0,\ldots,n-1$. Take any linearly independent $X_1, \ldots, X_k$ in $\fT_{n,\fC}^{-}$.
Suppose \begin{align*}
    \sum_{i=1}^k \alpha_iX_ie_1^{(2n)}=0, \quad \text{where} \quad \alpha_i \in \BF \quad \text{with} \quad i=1, \ldots,k.
\end{align*} Then also \begin{align*} 0&=\left(\left(S_n^T\right)^j\otimes I_2\right)\sum_{i=1}^k \alpha_iX_ie_1^{(2n)}=\sum_{i=1}^k \alpha_i\left(\left(S_n^T\right)^j\otimes I_2\right)X_i\left(e_1^{(n)} \otimes e_1^{(2)}\right)\\&=\sum_{i=1}^k \alpha_iX_i\left(\left(S_n^T\right)^j\otimes I_2\right)\left(e_1^{(n)} \otimes e_1^{(2)}\right)=\sum_{i=1}^k \alpha_iX_i\left(e_{j+1}^{(n)} \otimes e_1^{(2)}\right)\\&=\sum_{i=1}^k \alpha_iX_ie_{2j+1}^{(2n)} \quad \text{for} \quad j=0, \ldots, n-1.\end{align*} Hence all columns of $\sum_{i=1}^k \alpha_iX_i$ with indices $2j+1$ are zero. Now because of the relation that exists  between the entries of column $2j+1$ and $2j+2$, for all $j=0,\ldots,n-1$, it follows that all columns with indices $2j+2$ are also zero.
Therefore $\sum_{i=1}^k \alpha_iX_i=0$, which can only be if $\alpha_i=0$ for all $i=1, \ldots,k$, since $X_1, \ldots, X_k$ are linearly independent. This proves that for the basis vector $e_1^{(2n)} \in \BF^{2n}$, the vectors $X_1e_1^{(2n)}, \ldots, X_ke_1^{(2n)}$ are linearly independent in $\BF^{2n}$, therefore $\fT^-_{n,\fC}$ satisfies (C1).

The proof for (C2) goes analogously, with the required vector being the basis vector $e_{2n-1}^{(2n)} \in \BF^{2n}.$ The claims for $\fT^+_{n,\fC}$ again follow by the duality result of Lemma \ref{C1C2 duality} and the fact that $(\fT^+_{n,\fC})^*=\fT^-_{n,\fC}$
\end{proof}

Now we prove Theorem \ref{main result for double commutant}, first for $\BF=\BC$ and $\BF=\BR$ separately.

\begin{theorem}\label{main result for alg for C}
For any matrix $A \in \BC^{n \times n}$ 
the algebra $\left\{A\right\}_{\BC}''$ satisfies both conditions \tu{(C1)} and \tu{(C2)}.
\end{theorem}

\begin{proof}
[\bf Proof] Take $A \in \BC^{n \times n}$ with Jordan form as in \eqref{JordanComp}. By Lemma \ref{similar alg} we know it suffices to prove that the algebra $\{J_A\}_{\BC}''$ satisfies condition (C1). Using Lemma \ref{direct sum of alg} we need only prove that the matrix algebra $\fT^{+}_{\underline{n}_j,\BC}$ satisfies (C1) for all $j=1, \ldots r$. Equivalently, by Lemma \ref{sub alg}, we can show that $\diag\left(\fT^{+}_{n_{j,1},\BC}, \ldots,\fT^{+}_{n_{j,k_j},\BC}\right)$ satisfies (C1) for every $j=1, \ldots, r$, since $\fT^{+}_{\underline{n}_j,\BC}$ is a subalgebra  of $\diag\left(\fT^{+}_{n_{j,1},\BC}, \ldots,\fT^{+}_{n_{j,k_j},\BC}\right)$. Again by Lemma \ref{direct sum of alg} we need only prove that $\fT^{+}_{n_{j,q},\BC}$ satisfies (C1) for all $j=1, \ldots,r$ and $q=1, \ldots,k_j$, which is the case by Lemma \ref{alg lower triangular toeplitz} and Lemma \ref{C1C2 duality}.\\ \indent The proof for (C2) goes analogously, using the fact that $\fT^{+}_{n,\BC}$ satisfies (C2).
\end{proof}

\begin{theorem}\label{main result for alg for R}
For any matrix $A \in \BR^{n \times n}$ the algebra $\{A\}_{\BR}''$ satisfies both conditions \tu{(C1)} and \tu{(C2)}.
\end{theorem}

\begin{proof}[\bf Proof]
Take $A \in \BR^{n \times n}$ with Jordan form as in \eqref{JordanReal}. By Lemma \ref{similar alg} we know it suffices to prove that the algebra $\{J_A\}_{\BR}''$ satisfies condition (C1). Using Lemma \ref{direct sum of alg} we need only prove that the matrix algebras $\fT^{+}_{\un{m}_s,\fC}$ and $\fT^{+}_{\underline{n}_j,\BR}$ satisfy (C1) for all $j=1, \ldots r_2$ and $s=1, \ldots r_1$. Or equivalently, by Lemma \ref{sub alg} we can show that the algebras $\diag\left(\fT^{+}_{m_{s,1},\fC}, \ldots,\fT^{+}_{m_{s,l_s},\fC}\right)$ and $\diag\left(\fT^{+}_{n_{j,1},\BR}, \ldots,\fT^{+}_{n_{j,k_j},\BR}\right)$ satisfy (C1) for every $j=1, \ldots, r_2$ and $s=1,\ldots,r_1$, since $\fT^{+}_{\underline{m}_s,\fC}$ is a subalgebra  of $\diag\left(\fT^{+}_{m_{s,1},\fC}, \ldots,\fT^{+}_{m_{s,l_s},\fC}\right)$ and $\fT^{+}_{\underline{n}_j,\BR}$ is a subalgebra  of $\diag\left(\fT^{+}_{n_{j,1},\BR}, \ldots,\fT^{+}_{n_{j,k_j},\BR}\right)$. Again by Lemma \ref{direct sum of alg} we need only prove that $\fT^{+}_{m_{s,p},\fC}$ and $\fT^{+}_{n_{j,q},\BR}$ satisfy (C1) for all $s=1, \ldots, r_1$, $p=1,\ldots,l_s$, $j=1, \ldots,r_2$ and $q=1, \ldots,k_j$, which is the case by Lemma \ref{alg lower triangular toeplitz}, Lemma \ref{alg lower triangular toeplitz real} and Corollary \ref{duality transpose}. \\ \indent The proof for (C2) goes analogously, making use of the fact that $\fT^{+}_{n,\BR}$ and $\fT_{n,\fC}^{+}$ satisfy (C2).
\end{proof}

\begin{proof}[\bf Proof of Theorem \ref{main result for double commutant}]
This follows directly from Theorems \ref{main result for alg for C} and \ref{main result for alg for R}.
\end{proof}

We can now also prove our second main result given in the introduction.

\begin{proof}[\bf Proof of Theorem \ref{T:Main2}]
This follows directly from Theorems \ref{T:Lin} and \ref{main result for double commutant}.
\end{proof}

 \section{A Hill-Pick matrix criterium for the Lyapunov order}\label{S:HillPick}

In this section we return to the Lyapunov order. We prove Theorem \ref{main thm} and use this result together with Hill representations to determine an explicit matrix criterium for Lyapunov domination.

Throughout this section let $A,B\in \BF^{n \times n}$ with $A$ Lyapunov regular and $B\in\{A\}_{\BF}''$. Define the Lyapunov operators $\cL_A$ and $\cL_B$ by \eqref{LyapOp}, and denote their respective matricizations and Choi matrices by $L_A$ and $L_B$ and by $\BL_A$ and $\BL_B$. Since $A$ is Lyapunov regular, $\cL_A$ is invertible. Hence we can define the linear map $\cL_{A,B}$ as in \eqref{cL_AB}. Write $L_{A,B}$ and $\BL_{A,B}$ for the matricization and Choi matrices of $\cL_{A,B}$, respectively.

\begin{proof}[\bf Proof of Theorem \ref{main thm}]
It is easily verified that $\cL_A$ and $\cL_B$ are $*$-linear maps. By the comment on the bottom of page 62 in \cite{KMP00} we also know that $\cL_A^{-1}$ is $*$-linear, so that the composition $\cL_{A,B}$ of $\cL_A^{-1}$ and $\cL_B$ is also $*$-linear. Moreover, since $\cL_A$ is invertible, so is $L_A$ and $L_A^{-1}$ is the matricization of $\cL_A^{-1}$, so that $L_{A,B}=L_B L_{A}^{-1}$.

Proposition \ref{P:HillLBL} implies that the matricizations of $\cL_A$ and $\cL_B$ are given by
\begin{equation}\label{LALBform}
L_A=A^T\otimes I_n+I_n\otimes A^*,\quad
L_B=B^T\otimes I_n+I_n\otimes B^*.
\end{equation}
Since $B\in\{A\}_{\BF}''$, we have $B^*\in\{A^*\}_{\BF}''$. It then follows that $L_A$ and $L_B$ are both contained in the matrix algebra $\ov{\{A^*\}_{\BF}''}\otimes \{A^*\}_{\BF}''$. Moreover, $L_A^{-1}$ is in $\ov{\{A^*\}_{\BF}''}\otimes \{A^*\}_{\BF}''$ too. Indeed, since $\ov{\{A^*\}_{\BF}''}\otimes \{A^*\}_{\BF}''$ is an algebra that contains $L_A$, it also contains the bicommutant $\{L_A\}_{\BF}''$ of $L_A$, which in turn contains $L_A^{-1}$. We conclude that $L_{A,B}$ is contained in $\ov{\{A^*\}_{\BF}''}\otimes \{A^*\}_{\BF}''$, so that the proof is complete by an application of Theorem \ref{T:Main2}.
\end{proof}

As a consequence one can verify Lyapunov domination $A\leq_\cL B$, at least in the case covered by Theorem \ref{main thm}, by computing the Choi matrix of $\cL_{A,B}$. Using Hill representations and the associated Hill matrix it is possible to determine a more explicit criterium, in terms of the Jordan structure of $A$ and corresponding form of $B$ obtained from the fact that $B\in \{A\}_{\BF}''$, as described in Subsection \ref{SubS:bicommA}.

We shall conduct the computations for the case where $\BF=\BC$ explicitly. For this purpose, let the Lyapunov regular $A\in\BC^{n \times n}$ be given in Jordan form \eqref{JordanComp}, that is,
\begin{equation}\label{Jor decomp for L}
A=P\, J_A\, P^{-1},\quad \mbox{with }J_A=\diag\left(J_{\un{n}_1}\left(\la_1\right),\ldots,J_{\un{n}_r}\left(\la_r\right)\right)
\end{equation}
where $P\in\textup{GL}(n,\BC)$ and $\un{n}_j=\left(n_{j,1},\ldots,n_{j,k_j}\right)\in\BN^{k_j}$, ordered decreasingly, for some $k_j\in\BN$ and for $j=1,\ldots,r$, so that $\sum_{j=1}^r \sum_{i=1}^{k_j} n_{j,i}=n$. As explained in Subsection \ref{SubS:bicommA}, the matrix $B \in \{A\}_{\BC}''$ has the form
\begin{equation}\label{B decomp for L}
\begin{aligned}
B=P\widetilde{B}P^{-1}\quad\mbox{with}\quad\widetilde{B}=\diag\left(T_1, \ldots, T_r\right)\\
\mbox{and}\quad T_j=\diag\left(T_{j,1},\ldots,T_{j,k_j}\right) \in \fT_{\un{n}_j,\BC}^+,
\end{aligned}
\end{equation}
where the matrices $T_{j,q}$ are given by
\begin{equation}\label{Toep for B}
T_{j,q}=\sum_{i=0}^{n_{j,q}-1}t_{j,i}S_{n_{j,q}}^{i}.
\end{equation}
Note that the scalar coefficients $t_{j,i}$ do not depend on $q$. As a first step towards computing the Hill-Pick matrix $\BH_{A,B}$, we compute the matricization $L_{A,B}$.

\subsection{Computing $L_{A,B}$ for $\BF=\BC$}

First we establish the relation between $L_{A,B}$ and the analogue for the matrices $J_A$ and $\wtilB$, and similarly for $\BL_{A,B}$.

\begin{lemma}\label{L:LABBLAB}
The matricization $L_{A,B}$ and Choi matrix $\BL_{A,B}$ for $A$ and $B$ as above are given by
\begin{equation}\label{LABBLAB}
\begin{aligned}
L_{A,B} & =\left(P^T\otimes P^*\right)^{-1}L_{J_A,\widetilde{B}}\left(P^T \otimes P^*\right), \\
\BL_{A,B} & =\left(P^T\otimes P^{-1}\right)^*\BL_{J_A,\widetilde{B}}\left(P^T \otimes P^{-1}\right).
\end{aligned}
\end{equation}
Here $L_{J_A,\widetilde{B}}$ and $\BL_{J_A,\widetilde{B}}$ are the matricization and Choi matrix of the linear matrix map $\cL_{J_A,\wtilB}$ defined as in \eqref{cL_AB} with $A$ and $B$ replaced by $J_A$ and $\wtilB$, respectively.
\end{lemma}

\begin{proof}[\bf Proof]
By the Kronecker product property \eqref{kron prop} and the formulas for $L_A$ and $L_B$ in \eqref{LALBform} we get
\[
L_A=\left(P^T\otimes P^*\right)^{-1}L_{J_A}\left(P^T \otimes P^*\right),\ \ L_B=\left(P^T\otimes P^*\right)^{-1}L_{\wtilB}\left(P^T \otimes P^*\right).
\]
This gives
\begin{equation*}\begin{aligned}
    L_{A,B}&= \left(P^T \otimes P^*\right)^{-1}L_{\widetilde{B}}\left(P^T \otimes P^*\right)\left(\left(P^T \otimes P^*\right)^{-1}L_{J_A}\left(P^T \otimes P^*\right) \right)^{-1} \\ &= \left(P^T \otimes P^*\right)^{-1}L_{\widetilde{B}}L_{J_A}^{-1}\left(P^T \otimes P^*\right)=\left(P^T\otimes P^*\right)^{-1}L_{J_A,\widetilde{B}}\left(P^T \otimes P^*\right). \end{aligned}
\end{equation*}
The $4$-modularity property in \cite[Proposition 2.1]{P19} yields
\begin{equation*} 
\BL_{A,B}=\left(P^T\otimes P^{-1}\right)^*\BL_{J_A,\widetilde{B}}\left(P^T \otimes P^{-1}\right).\qedhere
\end{equation*}
\end{proof}

Next we compute $L_{A,B}$ via an explicit computation of $L_{J_A,\wtilB}$. Note that $L_{J_A,\wtilB}$ is in $\ov{\{J_A^*\}_{\BC}''} \otimes \{J_A^*\}_{\BC}''=\{J_A^T\}_{\BC}'' \otimes \{J_A^T\}_{\BC}''$, so that $L_{J_A,\wtilB}$ is of the form as described in Proposition \ref{P:SelfTensCdoubleCom}.

\begin{proposition}\label{constructing L_BL_A^{-1}}
For $A \in \BC^{n \times n}$ Lyapunov regular, with Jordan decomposition as in \eqref{Jor decomp for L} and $B \in \{A\}''_{\BC}$ as in \eqref{B decomp for L} and \eqref{Toep for B}, it follows that
\begin{equation*}
     L_{A,B} =\left(P^T \otimes P^*\right)^{-1}\diag \left(R_1, ... ,R_r \right)\left(P^T \otimes P^*\right) , \end{equation*} where \begin{equation*} R_j = \diag\left(R_{j,1}, ... ,R_{j,k_j}\right) \in \fT^{-}_{\un{n}_j,\{J_A^T\}_{\BC}''},  \end{equation*} with \begin{align} \label{R for}
    R_{j,q}&=\sum_{i=0}^{n_{j,q}-1}  \left(S_{n_{j,q}}^T \right)^{i}\otimes F_{j,i} \in \BC^{nn_{j,q} \times nn_{j,q}}\end{align} and \begin{align} \label{F for}  F_{j,i}=\sum_{l=0}^{i}(-1)^{i-l} D_{j,l}\left(\lambda_jI_n + J_A^* \right)^{l-i-1} \in \{J_A^T\}_{\BC}'';
\end{align} here $D_{j,0}=t_{j,0}I_n + \widetilde{B}^*$ and $ D_{j,l}=t_{j,l}I_n$ for $l=1, \ldots,n_{j,q}-1.$ Moreover, the matrix $F_{j,i}$ can be computed explicitly as
\begin{equation}\label{formula for scalars of D}
F_{j,i}=\bigoplus_{a=1}^r\bigoplus_{b=1}^{k_a} \sum_{c=0}^{n_{a,b}-1} f_{j,i}^{a,c} \left(S_{n_{a,b}}^T\right)^{c}
\end{equation}
where
\begin{equation}\label{formula for scalars of D2}
f_{j,i}^{a,c}=
\sum_{d=0}^{c} \binom{d+i}{d} \frac{(-1)^{d+i}\ov{t}_{a,c-d}}{\left(\la_j + \ov{\la}_a\right)^{d+i+1}}  + \sum_{l=0}^{i} \binom{c+i-l}{c} \frac{(-1)^{i-l+c} t_{j,l}}{\left(\la_j + \ov{\la}_a\right)^{c+i-l+1}}.
\end{equation}

\end{proposition}

\begin{proof}[\bf Proof]
By Lemma \ref{L:LABBLAB} we have
\[
L_{A,B} =\left(P^T\otimes P^*\right)^{-1}L_{J_A,\widetilde{B}}\left(P^T \otimes P^*\right),\quad\mbox{with}\quad L_{J_A,\widetilde{B}}=L_{\wtilB} L_{J_A}^{-1},
\]
where
\[
L_{J_A}=J_A^T\otimes I_n + I_n \otimes J_A^* \ands L_{\wtilB}=\wtilB^T\otimes I_n + I_n \otimes \wtilB^*.
\]
Note that the matrix $L_{J_A}$ can be written as
\begin{equation*}
\begin{aligned}L_{J_A}&=\diag\left(J_{\un{n}_1}\left(\la_1\right)^T,\ldots,J_{\un{n}_r}\left(\la_r\right)^T\right)\otimes I_n + I_n \otimes J_A^* \\&= \bigoplus_{j=1}^r \left(J_{\un{n}_j}\left(\lambda_j\right)^T \otimes I_n + I_{n_j} \otimes J_A^*\right).
\end{aligned}
\end{equation*}
Similarly the matrix $L_{\widetilde{B}}$ can be written as
\begin{equation*} \begin{aligned} L_{\widetilde{B}}&=\diag\left(T_1^T, ... ,T_r^T\right)\otimes I_n + I_n \otimes \widetilde{B}^* = \bigoplus_{j=1}^r \left(T_j^T \otimes I_n + I_{n_j} \otimes \widetilde{B}^*\right).
\end{aligned}
\end{equation*}
This gives
\begin{equation*}
L_{J_A,\widetilde{B}}=L_{\widetilde{B}}L_{J_A}^{-1}= \bigoplus_{j=1}^r \left(T_j^T \otimes I_n + I_{n_j} \otimes \widetilde{B}^*\right) \left(J_{\un{n}_j}\left(\lambda_j\right)^T \otimes I_n + I_{n_j} \otimes J_A^*\right) ^{-1}.
\end{equation*}
For $j=1,\ldots,r$ it follows that \begin{equation*} \begin{aligned}
    J_{\un{n}_j}\left(\lambda_j\right)^T \otimes I_n + I_{n_j} \otimes J_A^* &= \diag\left(J_{n_{j,1}}^T(\lambda_j), ... , J_{n_{j,k_j}}^T(\lambda_j)\right) \otimes I_n + I_{n_j}\otimes J_A^* \\ &= \bigoplus_{q=1}^{k_j}\left(J_{n_{j,q}}^T\left(\lambda_j\right) \otimes I_n + I_{n_{j,q}}\otimes J_A^*\right),\end{aligned}
\end{equation*} and \begin{equation*} \begin{aligned}
    T_j^T \otimes I_n + I_{n_j}\otimes  \widetilde{B}^* &= \diag\left(T_{j,1}^T, ... , T_{j,k_j}^T\right) \otimes I_n + I_{n_j} \otimes \widetilde{B}^* \\ &= \bigoplus_{q=1}^{k_j}\left(T_{j,q}^T \otimes I_n + I_{n_{j,q}} \otimes \widetilde{B}^*\right).\end{aligned}
\end{equation*}
For $j=1, \ldots, r$ and $q=1, \ldots, k_j$ it follows that
\begin{equation*}\begin{aligned} T_{j,q}^T \otimes I_n + I_{n_{j,q}} \otimes \widetilde{B}^* &=\sum_{i=0}^{n_{j,q}-1}t_{j,i}\left(S^T_{n_{j,q}}\right)^{i} \otimes I_n +I_{n_{j,q}} \otimes \widetilde{B}^* \\ &= \sum_{i=1}^{n_{j,q}-1} \left(S_{n_{j,q}}^T\right)^{i} \otimes t_{j,i}I_n + I_{n_{j,q}} \otimes \left(t_{j,0}I_n+ \widetilde{B}^*\right) \\ &=\sum_{i=0}^{n_{j,q}-1} \left(S_{n_{j,q}}^T\right)^i \otimes
D_{j,i}, \end{aligned} \end{equation*}
with $D_{j,0}$ and $D_{j,i}$, $i=1, \ldots, n_{j,q}-1$, as defined in the proposition.
Furthermore, for $j=1, \ldots, r$ and $q=1, \ldots, k_j$, it follows that
\begin{equation*}
    \begin{aligned} \left(J_{n_{j,q}}^T\left(\lambda_j\right) \otimes I_n + I_{n_{j,q}} \otimes J_A^*\right)^{-1} &= \left(S_{n_{j,q}}^T \otimes I_n + I_{n_{j,q}} \otimes \left( \lambda_j I_n + J_A^* \right)    \right)^{-1} \\ &= \sum_{l=0}^{n_{j,q}-1}(-1)^{l}\left(S_{n_{j,q}}^T\right)^l\otimes \left(\lambda_jI_n + J_A^* \right)^{-l-1}.  \end{aligned}
\end{equation*}
Thus for $j=1,\ldots,r$ and $q=1,\ldots,k_j$ we have
\begin{equation*}
    \begin{aligned} &\left(T_{j,q}^T \otimes I_n + I_{n_{j,q}} \otimes \tilde{B}^*\right)\left(J_{n_{j,q}}^T\left(\lambda_j\right) \otimes I_n + I_{n_{j,q}}\otimes J_A^*\right)^{-1}\\ &\quad = \left(\sum_{i=0}^{n_{j,q}-1} \left(S_{n_{j,q}}^T\right)^i \otimes D_{j,i}\right)\left(\sum_{l=0}^{n_{j,q}-1}(-1)^{l}\left(S_{n_{j,q}}^T\right)^l\otimes \left(\lambda_jI_n + J_A^* \right)^{-l-1} \right) \\ &\quad=  \sum_{i=0}^{n_{j,q}-1}\sum_{l=0}^{n_{j,q}-1}(-1)^l\left( \left(S_{n_{j,q}}^T\right)^i \otimes D_{j,i} \right)\left(\left(S_{n_{j,q}}^T\right)^l\otimes \left(\lambda_jI_n + J_A^* \right)^{-l-1}\right) \\ &\quad =\sum_{i=0}^{n_{j,q}-1} \sum_{l=0}^{n_{j,q}-1}(-1)^l \left(S_{n_{j,q}}^T \right)^i\left(S_{n_{j,q}}^T\right)^l \otimes D_{j,i}  \left(\lambda_jI_n + J_A^* \right)^{-l-1} \\ &\quad = \sum_{i=0}^{n_{j,q}-1} \sum_{l=0}^{n_{j,q}-1}(-1)^l \left(S_{n_{j,q}}^T \right)^{i+l}\otimes D_{j,i}  \left(\lambda_jI_n + J_A^* \right)^{-l-1} \\ &\quad =  \sum_{i=0}^{n_{j,q}-1}\left(S_{n_{j,q}}^T\right)^i \otimes \sum_{l=0}^i(-1)^{i-l} D_{j,l}  \left(\lambda_jI_n + J_A^* \right)^{l-i-1}=R_{j,q}.
    \end{aligned}
\end{equation*}
Hence \begin{equation*}
    L_{A,B}=\left(P^T \otimes P^*\right)^{-1}\bigoplus_{j=1}^r\bigoplus_{q=1}^{k_j}R_{j,q}\left(P^T \otimes P^*\right).
\end{equation*}
Since $\{J_A^T\}_{\BC}''=\diag\left(\fT_{\un{n}_1,\BC}^-,\dots,\fT_{\un{n}_r,\BC}^-\right)$ is an algebra which is closed under inversion and conjugation, i.e., $\{J_A^T\}_{\BC}''=\{J_A^*\}_{\BC}''$,  it follows that for $j=1,\ldots,r$ and $i=0,\ldots,n_{j,q}-1$:
\[
F_{j,i}=\sum_{l=0}^{i}(-1)^{i-l} D_{j,l}  \left(\lambda_jI_n + J_A^* \right)^{l-i-1} \in \{J_A^T\}_{\BC}''
\]
From the fact that $F_{j,i}$ does not depend on $q,$
we see that \begin{equation*} R_{j,q}=\sum_{i=0}^{n_{j,q}-1}\left(S_{n_{j,q}}^T\right)^{i} \otimes F_{j,i} \end{equation*} is the $nn_{j,q} \times nn_{j,q}$ left upper corner of
\begin{equation*} R_{j,1}=\sum_{i=0}^{n_{j,1}-1}\left(S_{n_{j,1}}^T\right)^{i} \otimes F_{j,i},
\end{equation*}
for $j=1,\ldots,r$ and $q=1,\ldots,k_j$. Taking all of the above into consideration gives
\begin{equation*}
L_{A,B} =\left(P^T \otimes P^*\right)^{-1}\diag \left(R_1, ... ,R_r \right)\left(P^T \otimes P^*\right),
\end{equation*}
with $R_j$ as defined in the proposition.

It remains to determine the explicit formula for $F_{j,i}$. Note first that
\begin{align}
F_{j,i}&=\sum_{l=0}^{i}(-1)^{i-l} D_{j,l}  \left(\lambda_jI_n + J_A^* \right)^{l-i-1} \notag  \\
&= (-1)^iD_{j,0}\left(\lambda_jI_n + J_A^* \right)^{-i-1} + \sum_{l=1}^{i}(-1)^{i-l} D_{j,l}  \left(\lambda_jI_n + J_A^* \right)^{l-i-1} \notag \\
&=(-1)^i \left(t_{j,0}I_n+\widetilde{B}^*\right)\left(\lambda_jI_n + J_A^* \right)^{-i-1} + \sum_{l=1}^{i}(-1)^{i-l} t_{j,l}  \left(\lambda_jI_n + J_A^* \right)^{l-i-1} \notag \\
&=(-1)^i\widetilde{B}^*\left(\lambda_jI_n + J_A^* \right)^{-i-1} + \sum_{l=0}^{i}(-1)^{i-l} t_{j,l} \left(\lambda_jI_n + J_A^* \right)^{l-i-1},\label{Fji-mid}
\end{align}
with
\begin{align*} \widetilde{B}^*=\bigoplus_{a=1}^r\bigoplus_{b=1}^{k_a} T^*_{a,b}=\bigoplus_{a=1}^r\bigoplus_{b=1}^{k_a} \sum_{c=0}^{n_{a,b}-1} \ov{t}_{a,c}\left(S_{n_{a,b}}^T\right)^{c}.
\end{align*}
Furthermore, by making use of the formula for powers of the inverse of a Jordan block, cf., \cite[Page 70]{G14}, it follows for $i=0, \ldots, n_{j,q}-1$ that
\begin{align*}
\left(\lambda_jI_n + J_A^* \right)^{-i-1}&=\left(\bigoplus_{a=1}^r\bigoplus_{b=1}^{k_a}J_{n_{a,b}}^T\left(\la_j+\ov{\la}_a\right)  \right)^{-i-1}\\
&=\bigoplus_{a=1}^r\bigoplus_{b=1}^{k_a}\left(J_{n_{a,b}}\left(\la_j+\ov{\la}_a\right)^{-i-1}\right)^T \\
&=\bigoplus_{a=1}^r\bigoplus_{b=1}^{k_a}\left( \sum_{d=0}^{n_{a,b}-1}\binom{d+i}{d} \frac{(-1)^{d}}{\left(\la_j + \ov{\la}_a\right)^{d+i+1}}
S_{n_{a,b}}^d\right)^T\\
&=\bigoplus_{a=1}^r\bigoplus_{b=1}^{k_a} \sum_{d=0}^{n_{a,b}-1}\binom{d+i}{d} \frac{(-1)^{d}}{\left(\la_j + \ov{\la}_a\right)^{d+i+1}}\left(S_{n_{a,b}}^T\right)^d.
\end{align*}
Therefore, we have
\begin{align*}
&(-1)^i\widetilde{B}^*\left(\lambda_jI_n + J_A^* \right)^{-i-1}=\\
&\qquad = \bigoplus_{a=1}^r\bigoplus_{b=1}^{k_a} \left( \sum_{c=0}^{n_{a,b}-1} \ov{t}_{a,c}\left(S_{n_{a,b}}^T\right)^{c}\right) \left( \sum_{d=0}^{n_{a,b}-1}\binom{d+i}{d} \frac{(-1)^{d+i}}{\left(\la_j + \ov{\la}_a\right)^{d+i+1}}\left(S_{n_{a,b}}^T\right)^d\right)\\
&\qquad =  \bigoplus_{a=1}^r\bigoplus_{b=1}^{k_a} \sum_{c=0}^{n_{a,b}-1}  \sum_{d=0}^{n_{a,b}-1} \binom{d+i}{d} \frac{(-1)^{d+i}\ov{t}_{a,c}}{\left(\la_j + \ov{\la}_a\right)^{d+i+1}}\left(S_{n_{a,b}}^T\right)^{c+d}\\
&\qquad= \bigoplus_{a=1}^r\bigoplus_{b=1}^{k_a} \sum_{c=0}^{n_{a,b}-1}  \sum_{d=0}^{c} \binom{d+i}{d} \frac{(-1)^{d+i}\ov{t}_{a,c-d}}{\left(\la_j + \ov{\la}_a\right)^{d+i+1}}\left(S_{n_{a,b}}^T\right)^{c}
\end{align*}
and
\begin{align*}
&\sum_{l=0}^{i}(-1)^{i-l} t_{j,l} \left(\lambda_jI_n + J_A^* \right)^{l-i-1} = \\
&\qquad = \bigoplus_{a=1}^r\bigoplus_{b=1}^{k_a}  \sum_{l=0}^{i} t_{j,l} \sum_{c=0}^{n_{a,b}-1} \binom{c+i-l}{c}\frac{(-1)^{i-l+c}}{ \left(\la_j + \ov{\la}_a\right)^{c+i-l+1}}
\left(S_{n_{a,b}}^T\right)^c\\
&\qquad = \bigoplus_{a=1}^r\bigoplus_{b=1}^{k_a} \sum_{c=0}^{n_{a,b}-1}   \sum_{l=0}^{i} \binom{c+i-l}{c}\frac{(-1)^{i-l+c}t_{j,l}}{ \left(\la_j + \ov{\la}_a\right)^{c+i-l+1}}
\left(S_{n_{a,b}}^T\right)^c.
\end{align*}
Inserting the end results of the last two computations into \eqref{Fji-mid} yields \eqref{formula for scalars of D} and \eqref{formula for scalars of D2}.
\end{proof}

Via computations similar to those in the above proof one can show directly that
\begin{equation}\label{reflex}
f_{j,i}^{a,c}= \ov{f}_{a,c}^{j,i}
\end{equation}
which confirms item (iii) of Theorem \ref{T:*-linear}.

\subsection{Computing $\BH_{A,B}$ for $\BF=\BC$}

The Hill matrix of the $*$-linear map $\cL_{A,B}$ depends on a choice of matrices $L_1,\ldots,L_m$ so that their span corresponds to the span of the block entries of the matricization $L_{A,B}$. Via the relation between $L_{A,B}$ and $L_{J_A,\wtilB}$ obtained in Lemma \ref{L:LABBLAB}, where $L_{J_A,\wtilB}$ is the matricization of $\cL_{J_A,\wtilB}$, it follows that we just as well make a selection of the block entries of $L_{J_A,\wtilB}$, with the $L_1,\ldots,L_m$ being determined by \eqref{LABBLAB}. Using the Jordan structure of $A$ given by \eqref{Jor decomp for L}, for $L_{J_A,\wtilB}$ there is a natural choice of block entries to select, based on the formula in \eqref{formula for transpose C}, and it is the Hill matrix with respect to this choice that we will call the Hill-Pick matrix, denoted $\BH_{A,B}$, associated with $A$ and $B$. This selection of the block entries of $L_{J_A,\wtilB}$ consists of $w=n_{1,1}+\cdots + n_{r,1}$ entries and corresponds to the selection of indices given by
\begin{align*}
\Upsilon=&\left\{(1,1),\ldots,(n_{1,1},1),(n_1+1,n_1+1),\ldots,(n_1 +n_{2,1},n_1+1),\ldots\right.\\
&\qquad\qquad\left. \ldots,(n_1+\cdots+ n_{r-1}+1,n_1+\cdots+n_{r-1}+1),\ldots\right.\\
&\qquad\qquad\qquad\qquad\left. \ldots(n_1+\cdots+n_{r-1}+n_{r,1},n_1+\cdots+n_{r-1}+1) \right\},
\end{align*}
where for $j=1,\ldots,r$ we define $n_j=n_{j,1}+\cdots+n_{j,k_j}$. In the ``Toeplitz structure with possible repetition'' of $\{J_A^T\}_{\BC}'' \otimes \{J_A^T\}_{\BC}''$ as described in \eqref{formula for transpose C}, in which $L_{J_A,\wtilB}$ is included, the $j$-th set of indices, running from $(n_1+\cdots+ n_{j-1}+1,n_1+\cdots+n_{j-1}+1)$ to $(n_1+\cdots+ n_{j-1}+n_{j,1},n_1+\cdots+n_{j-1}+1)$, corresponds to taking the matrices in the first column of each of the lower triangular block Toeplitz matrices with possible repetition coming from $\fT^-_{\un{n}_j,\BC} \otimes \{J_A^T\}_{\BC}'' = \fT^-_{\un{n}_j,\{J_A^T\}_{\BC}''}$. 

We note here that the number $w$ may be larger than the rank of the Choi matrix, $m=\rank \BL_{A,B}$, so that $\BH_{A,B}$ may be singular. Our results on non-minimal Hill representations in Subsection \ref{SubS:NonMinHill}, in particular Theorem \ref{T:Hill non-min}, show that even in this case positive semidefiniteness of $\BH_{A,B}$ remains a necessary and sufficient criterium to check whether $\cL_{A,B}$ is (completely) positive, and hence whether $B$ Lyapunov dominates $A$.

For the matrix representation of $\BH_{A,B}$ it makes sense to divide the set $\Upsilon$ into $r$ disjoint subsets
\begin{align*}
&\Upsilon=\bigcap_{j=1}^r \Upsilon_j \quad \mbox{with}\ \
\Upsilon_j:=\{(n_1+\cdots+n_{j-1}+1,n_1+\cdots+n_{j-1}+1),\ldots\\
& \qquad\qquad\ldots , (n_1+\cdots+ n_{j-1}+n_{j,1},n_1+\cdots+n_{j-1}+1) \},
\end{align*}
and then divide $\BH_{A,B}$ accordingly
\begin{equation}\label{HillPick1}
\BH_{A,B}=\mat{\BH_{A,B}^{(i,j)}}_{i,j=1}^r \quad \mbox{with}\quad
\BH_{A,B}^{(i,j)}\in \BC^{n_{i,1} \times n_{j,1}}
\end{equation}
the matrix determined by the indices of $\Upsilon_i$ and $\Upsilon_j$.

Now note that the $(n_1+\cdots+n_{i-1}+a, n_1+\cdots+n_{i-1}+1)$-th block of $L_{A,B}$ corresponds to $F_{i,a-1}$, and within $F_{i,a-1}$ the $(n_1+\cdots+n_{j-1}+b, n_1+\cdots+n_{j-1}+1)$-th entry is equal to $f_{i,a-1}^{j,b-1}$. Hence, we have
\begin{equation}\label{HillPick2}
\BH_{A,B}^{(i,j)}=\mat{f_{j,b-1}^{i,a-1}}_{a=1,\ldots,n_{i,1}}^{b=1,\ldots,n_{j,1}},\quad\mbox{for }i,j=1,\ldots,r,
\end{equation}
with the matrix entries defined as in \eqref{formula for scalars of D2}. Note that \eqref{reflex} confirms that $\BH_{A,B}$ is Hermitian. In conclusion, we have proved the following result.
\begin{theorem}\label{T:HillPickmatrix}
Let $A,B\in\BC^{n \times n}$ with $A$ Lyapunov regular and $B\in\{A\}_{\BC}''$. Then the Hill-Pick matrix $\BH_{A,B}$ defined above is given by \eqref{HillPick1} and \eqref{HillPick2} with entries as in \eqref{formula for scalars of D2}. In particular, $B$ Lyapunov dominates $A$ if and only if the matrix $\BH_{A,B}$ is positive semidefinite.
\end{theorem}

Next we consider two simple examples, starting with the case where $A$ is a diagonal matrix.

\begin{example}
Let $A=\diag \left(\la_1, \ldots, \la_n\right)\in \BC^{n \times n}$ be Lyapunov regular, meaning $\la_i+\ov{\la}_j\neq 0$ for all $i,j$. Then $B \in \{A\}''_{\BC}$ is of the form \[B=\diag \left(t_1, \ldots, t_n\right) \quad \text{where} \quad t_i=t_j \quad \text{whenever} \quad \la_i=\la_j.\]
For $L_{A,B}$ it follows that
\begin{align*}L_{A,B}&=\diag\left(L_1, \ldots, L_n\right) \quad \text{where} \quad L_i=\diag\left(\ell_{i,1}, \ldots, \ell_{i,n}\right) \in\BC^{n \times n} \quad \text{and} \\ \ell_{i,j}&:=\frac{\ov{t}_j+t_i}{\ov{\la}_j+\la_i} \quad \text{ for all } i,j=1, \ldots,n. \end{align*}
In case $\la_i \ne \la_j$ for all $i,j$, the matrix $\BH_{A,B}$ is equal to the Pick matrix of \cite{YS67}
with $(i,j)$-th entry given by
\[\frac{\ov{t}_i+t_j}{\ov{\la}_i+\la_j} \quad \text{for} \quad i,j=1, \ldots,n.\]
For the case where $\la_i=\la_j$ for some $i,j=1, \ldots,n$ it follows that the linearly dependent rows and corresponding columns in the Pick matrix are removed to produce $\BH_{A,B}.$
\end{example}

Next we consider the case of a single Jordan block.

\begin{example}
Take $A=J_n(\la)$ with $\la$ having a nonzero real part. Then $A$ is Lyapunov regular. We have $\left\{A\right\}''_{\BC}=\fT_{n,\BC}^{+}$. Hence $B \in \left\{A\right\}_{\BC}''$ has the form $B=\sum_{k=0}^{n-1}t_kS_n^{k}$ for $t_k\in\BC$. Following Proposition \ref{constructing L_BL_A^{-1}}, where the indices $j$ and $q$ do not exist, we get
\begin{align*}
L_{A,B}&=\sum_{i=0}^{n-1}\left(S_n^T\right)^{i} \otimes F_{i} \quad \text{where} \quad F_i=\sum_{l=0}^{i}(-1)^{i-l}D_{l}\left(\la I_n + A^*\right)^{l-i-1};
\end{align*}
here $D_0=t_0I_n+B^*$ and $D_l=t_{l}I_n$ for $l=1,\ldots, n-1$. Making use of the formulas in \eqref{formula for scalars of D} and \eqref{formula for scalars of D2} yields
\begin{align*}
F_i&=\sum_{c=0}^{n-1}\left( \sum_{d=0}^{c} \binom{d+i}{d}\frac{(-1)^{d+i}\ov{t}_{c-d}} {\left(\la + \ov{\la}\right)^{d+i+1}} \right. \\  &\left.\qquad \qquad \qquad + \sum_{l=0}^{i} \binom{c+i-l}{c} \frac{(-1)^{i-l+c}t_{l}}{\left(\la + \ov{\la}\right)^{c+i-l+1}}\right)\left(S_n^T\right)^{c}.
\end{align*}

\noindent Following the construction of $\BH_{A,B}$ above gives
\begin{align*} \BH_{A,B} :=\left[\begin{array}{ccccccccccc} F_0e_1 & F_1e_1 & \hdots & F_{n-1}e_1  \end{array}\right] \in \BC^{n \times n},
\end{align*}
with $p,q$-th entry given by (set $i+1=q$ and $c=p-1$)
\begin{align*}
\left[\BH_{A,B}\right]_{pq}&= \sum_{d=0}^{p-1} \binom{d+q-1}{d}\frac{(-1)^{d+q-1}\ov{t}_{p-1}} {\left(\la + \ov{\la}\right)^{d+q}} \\ &\qquad \qquad \qquad \qquad + \sum_{l=0}^{q-1} \binom{p+q-2-l}{p-1} \frac{(-1)^{q-l+p-2}t_{l}}{\left(\la + \ov{\la}\right)^{p+q-l-1}}.
\end{align*}
\end{example}

\subsection{The case $\BF=\BR$}
When $\BF=\BR$ a similar strategy can be followed as for the complex case, albeit more complicated because one has to distinguish between real and pairs of complex eigenvalues. We present here only the basic details, without going through the computation done in the previous subsection for the case $\BF=\BC$.

Let the Lyapunov regular $A \in \BR^{n \times n}$ be given in Jordan form \eqref{JordanReal}, that is, with (distinct) complex eigenvalues $\la_1,\ldots,\la_{r_1}$ and (distinct) real eigenvalues $\tau_1,\ldots,\tau_{r_2}$, the decomposition of $A$ of the form $A=P\,J_A\,  P^{-1},$ where
\begin{equation}\label{JordanReal for L}
 J_A=\diag\left(J_{\un{m}_1}\left(C_{\la_1}\right),\ldots,J_{\un{m}_{r_1}}\left(C_{\la_{r_1}}\right),J_{\un{n}_1}\left(\tau_1\right),\ldots , J_{\un{n}_{r_2}}\left(\tau_{r_2}\right)\right)
\end{equation}
with $P\in\textup{GL}(n,\BR)$ and $\un{n}_j=\left(n_{j,1},\ldots,n_{j,k_j}\right)\in\BN^{k_j}$, $\un{m}_s=\left(m_{s,1},\ldots,m_{s,l_s}\right)\in \BN^{l_s}$, both ordered decreasingly, for some $k_j,l_s\in\BN$ and for $s=1,\ldots,r_1$ and $j=1,\ldots,r_2$. Now let $B\in\{A\}_\BR''$, so that $B$ is of the form described in Subsection \ref{SubS:bicommA}, that is,
\[
B= P \wtilB P^{-1}=
 P\,\diag\left(R_1,\ldots,R_{r_1},T_1,\dots , T_{r_2}\right)\,  P^{-1}
\]
with, for $s=1,\ldots,r_1$ and $j=1,\ldots,r_2$,
\begin{align*}
R_s=\diag(R_{s,1},\ldots,R_{s,l_s})\in \fT^+_{\un{m}_s,\fC} \ands
T_j=\diag(T_{j,1},\ldots,R_{j,k_j})\in \fT^+_{\un{n}_j,\BR},
\end{align*}
given by, for $p=1,\ldots,l_s$ and $q=1,\ldots,k_j$,
\begin{align*}
R_{s,p}=\sum_{v=0}^{m_{s,p}-1}S_{m_{s,p}}^v \otimes C_{\la_{s,v}}, \ands
T_{j,q}=\sum_{v=0}^{n_{j,q}-1} t_{j,v} S_{n_{j,q}}^v,
\end{align*}
where $\la_{s,v}=\al_{s,v}+ i \be_{s,v}\in\BC$ and $t_{j,v}\in\BR$.

Again, the Hill matrix of the $*$-linear map $\cL_{A,B}$ depends on a choice of matrices $L_1,\ldots,L_m$ so that their span corresponds to the span of the block entries of the matricization $L_{A,B}$, which we can select based on the Jordan structure of $A\in\BR^{n\times n}$, working with $L_{J_A,\wtilB}$ rather that with $L_{A,B}$. In this case the selection of the block entries of $L_{J_A,\wtilB}$ consists of $w=2m_{1,1}+\cdots+2m_{r_1,1}+n_{1,1}+\cdots + n_{r_2,1}$ entries and corresponds to the selection of indices given by
\begin{align*}
\Upsilon=&\left\{(1,1),\ldots,(2m_{1,1},1),(2m_1+1,2m_1+1),\ldots,(2m_1 +2m_{2,1},2m_1+1),\ldots\right.\\
&\qquad\left. \ldots,(2m_1+\cdots+ 2m_{r_1-1}+1,2m_1+\cdots+2m_{r_1-1}+1),\ldots\right.\\
&\qquad\qquad\qquad\left. \ldots(2m_1+\cdots+2m_{r_1-1}+2m_{r_1,1},2m_1+\cdots+2m_{r_1-1}+1),\right.\\&\qquad\qquad\qquad\left. (1,1),\ldots,(n_{1,1},1),(n_1+1,n_1+1),\ldots,(n_1 +n_{2,1},n_1+1),\ldots\right.\\
&\qquad\qquad\left. \ldots,(n_1+\cdots+ n_{r_2-1}+1,n_1+\cdots+n_{r_2-1}+1),\ldots\right.\\
&\qquad\qquad\qquad\qquad\left. \ldots(n_1+\cdots+n_{r_2-1}+n_{r_2,1},n_1+\cdots+n_{r_2-1}+1)  \right\},
\end{align*}
where for $s=1,\ldots,r_1$ and $j=1,\ldots,r_2$ we define $m_s=m_{s,1}+\cdots+m_{s,l_s}$ and $n_j=n_{j,1}+\cdots+n_{j,k_j}$. As in the complex case, we may have $w> m=\rank \BL_{A,B}$, but the Hill-Pick matrix criterium remains both necessary and sufficient.

For the matrix representation of $\BH_{A,B}$ it makes sense to divide the set $\Upsilon$ into $r:=r_1+r_2$ disjoint subsets \begin{align*}
\Upsilon&= \Upsilon_1 \bigcup \Upsilon_2 \quad \text{where} \quad
\Upsilon_1=\bigcup_{i=1}^{r_1} \Upsilon_i^{(1)} \ands \Upsilon_2=\bigcup_{i=1}^{r_2} \Upsilon_i^{(2)} \quad \mbox{with} \\
\Upsilon_i^{(1)}&:=\{(2m_1+\cdots+ 2m_{i-1}+1,2m_1+\cdots+2m_{i-1}+1),\ldots\\
& \qquad \qquad \quad \qquad \ldots , (2m_1+\cdots+ 2m_{i-1}+2m_{i,1},2m_1+\cdots+2m_{i-1}+1) \},\\ \Upsilon_i^{(2)}&:=\{(n_1+\cdots+ n_{i-1}+1,n_1+\cdots+n_{i-1}+1),\ldots\\
& \qquad \quad \qquad\qquad\qquad \qquad  \ldots , (n_1+\cdots+ n_{i-1}+n_{i,1},n_1+\cdots+n_{i-1}+1) \}.
\end{align*}
We can then divide $\BH_{A,B}$ according to this partition of $\Up$ as
\begin{align}\label{HillPickreal}
\BH_{A,B}&=\begin{bmatrix} {}^1\BH_{A,B} & {}^2\BH_{A,B}\\ {}^2\BH_{A,B}^T & {}^3\BH_{A,B}\end{bmatrix} \in \BR^{w \times w},\ \mbox{where}\\
{}^1\BH_{A,B}=\mat{{}^1\BH_{A,B}^{(i,j)}}_{i,j=1}^{r_1},\quad
{}^2\BH_{A,B}&=\mat{{}^2\BH_{A,B}^{(i,j)}}_{i=1, \ldots, r_1}^{j=1,\ldots,r_2},\quad
{}^3\BH_{A,B}=\mat{{}^3\BH_{A,B}^{(i,j)}}_{i,j=1}^{r_2},\notag \\
\mbox{ with }\ {}^1\BH_{A,B}^{(i,j)}\in \BR^{2m_{i,1} \times 2m_{j,1}},&\quad
{}^2\BH_{A,B}^{(i,j)}\in \BR^{m_{i,1} \times n_{j,1}},\quad
{}^3\BH_{A,B}^{(i,j)}\in \BR^{n_{i,1} \times n_{j,1}}\notag
\end{align}
the matrix determined by the indices of $\Upsilon_i^{(1)}$ and $\Upsilon_j^{(1)}$, of $\Upsilon_i^{(1)}$ and $\Upsilon_j^{(2)}$, and of $\Upsilon_i^{(2)}$ and $\Upsilon_j^{(2)}$, respectively. The scalar entries of the matrices ${}^1\BH_{A,B}^{(i,j)}$, ${}^2\BH_{A,B}^{(i,j)}$ and ${}^3\BH_{A,B}^{(i,j)}$ can again be computed explicitly via an explicit computation of $L_{A,B}$, similarly, though significantly more laborious, as was done in Theorem \ref{constructing L_BL_A^{-1}} for the complex case, and then choosing the entries according to the selection of indices in $\Upsilon$.

\begin{theorem}\label{T:HillPickmatrixreal}
Let $A,B\in\BR^{n \times n}$ with $A$ Lyapunov regular and $B\in\{A\}_{\BR}''$. Then the Hill-Pick matrix $\BH_{A,B}$ defined above is given by \eqref{HillPickreal} 
In particular, $B$ Lyapunov dominates $A$ if and only if the matrix $\BH_{A,B}$ is positive semidefinite.
\end{theorem}

\subsection{The Stein order}

As mentioned in the introduction, there is an analogous matrix order associated with Nevanlinna-Pick interpolation in the unit disk, based on the Stein inequality; cf., Remark 2.6 in \cite{BtH10}. For $A\in\BF^{n \times n}$, the Stein inequality asks for a $H\in \cH_n$ so that $H-AHA^*\in \ov{\cP}_n$. Associated with this is the $*$-linear map $\cL_A(X)=X-A X A^*$, which has a matricization $L_A=I_{n^2}-\ov{A} \otimes A$ and is invertible whenever the eigenvalues $\la_1,\ldots,\la_n$ of $A$ satisfy $\la_i\ov{\la}_j\neq 0$, $i,j=1\ldots,n$. For $A,B\in\BF^{n \times n}$ so that $\cL_A$ is invertible, one then says that $B$ Stein dominates $A$ if each Stein solution $H\in \cH_n$ of $A$ is also a Stein solution of $B$, or equivalently, $\cL_B \cL_A^{-1}$ is a positive map. If one further restricts to $B\in\{A\}_{\BF}''$, reasoning as in the proof of Theorem \ref{main thm} it follows that $L_A$, $L_B$, $L_A^{-1}$ and consequently $L_B L_A^{-1}$ are all in $\ov{\{A\}_{\BF}''}\otimes \{A\}_{\BF}''$, leading to the observation that $\cL_B \cL_A^{-1}$ is a positive map precisely when it is completely positive. Hence, checking positive semidefiniteness of the corresponding Choi matrix provides as necessary and sufficient criteria for Stein domination and one can further employ Hill representations to determine a more efficient matrix criteria, which will correspond to the classical Pick matrix criteria in the case of the diagonal matrices.

\medskip

\paragraph{\bf Acknowledgments}
This work is based on research supported in part by the National Research Foundation of South Africa (NRF) and the DSI-NRF Centre of Excellence in Mathematical and Statistical Sciences (CoE-MaSS). Any opinion, finding and conclusion or recommendation expressed in this material is that of the authors and the NRF and CoE-MaSS do not accept any liability in this regard.


\end{document}